\documentclass[review,onefignum,onetabnum]{siamart171218}

\usepackage{amsmath}
\usepackage{amssymb,mathrsfs}
\usepackage{graphics}
\usepackage{graphicx}
\usepackage{epstopdf}
\usepackage{verbatim}
\usepackage[colorinlistoftodos,disable]{todonotes}

\newtheorem{thm}{Theorem}[section]

\newtheorem{rmk}{Remark}[section]

\ifpdf
  \DeclareGraphicsExtensions{.eps,.pdf,.png,.jpg}
\else
  \DeclareGraphicsExtensions{.eps}
\fi

\newsiamremark{remark}{Remark}
\newsiamremark{hypothesis}{Hypothesis}
\crefname{hypothesis}{Hypothesis}{Hypotheses}
\newsiamthm{claim}{Claim}

\headers{Linearized inverse Schr\"{o}dinger potential problem}
{V. Isakov, S. Lu and B. Xu}

\title{Linearized inverse Schr\"{o}dinger potential problem at a large wavenumber\thanks{Submitted to the editors DATE.
\funding{This research is supported in part by the Emylou Keith and Betty Dutcher Distinguished Professorship and the NSF grant DMS 15-14886.
Shuai Lu is supported by NSFC No.11522108, Shanghai Municipal Education Commission No.16SG01.
Boxi Xu is supported by NSFC No.11801351.}}}

\author{
  Victor Isakov\thanks{Department of Mathematics, Statistics, and Physics, Wichita State University, Wichita, KS 67260-0033, USA (Email: victor.isakov@wichita.edu).}
  \and
  Shuai Lu\thanks{School of Mathematical Sciences, Fudan University, No.220 Road Handan, Shanghai 200433, China (Email: slu@fudan.edu.cn).}
  \and
  Boxi Xu\thanks{Corresponding author. School of Mathematics, Shanghai University of Finance and Economics, No.777 Road Guoding, Shanghai 200433, China (Email: xu.boxi@mail.sufe.edu.cn).}
}

\usepackage{amsopn}

\ifpdf
\hypersetup{
  pdftitle={Linearized inverse Schr\"{o}dinger potential problem with a large wavenumber},
  pdfauthor={V. Isakov, S. Lu and B. Xu}
}
\fi

\begin{document}

\maketitle

\begin{abstract}
We investigate recovery of the (Schr\"{o}dinger) potential function from many boundary measurements at a large wavenumber. By considering such a linearized form, we obtain a H\"{o}lder type stability which is a big improvement over a logarithmic stability in low wavenumbers. Furthermore we extend the discussion to the linearized inverse Schr\"{o}dinger potential problem with attenuation, where an exponential dependence of the attenuation constant is traced in the stability estimate.
Based on the linearized problem, a reconstruction algorithm is proposed aiming at the recovery of the Fourier modes of the potential function. By choosing the large wavenumber appropriately, we verify the efficiency of the proposed algorithm by several numerical examples.
\end{abstract}




\section{Introduction}\label{se_intro}

We consider a linearized problem of recovering the potential function $c: = c(x)$ in the Schr\"{o}dinger equation
\begin{equation}\label{eqn:mainprob}
\left\{
\begin{aligned}
- \Delta u - (k^{2} - c) u &= 0 & & \text{in\ } \Omega \subset \mathbb{R}^{n},\\
u &= g_{0} & & \text{on\ } \partial \Omega,
\end{aligned}
\right.
\end{equation}
from many boundary measurements at a large wavenumber $k$ in the dimension $n \geqslant 2$.
As the data, we use a linearized form of the standard Dirichlet-to-Neumann (DtN) map
\begin{equation*}
\Lambda: g_{0} \mapsto g_{1} := \partial_{\nu} u \quad \text{on\ } \partial \Omega
\end{equation*}
which will be specified below.

Reconstruction of the Schr\"{o}dinger potential function $c(x)$ at a fixed wavenumber $k = 0$ in (\ref{eqn:mainprob}) retrospects back to the original inverse conductivity problem, proposed by Calder\'{o}n, arising in electrical impedance tomography where uniqueness for the linearization can be proven by using the complex exponential solutions \cite{C1980}. Later on, the fundamental work by Sylvester and Uhlmann proved the global uniqueness for the inverse potential problem in three and higher dimensions by constructing almost complex exponential solutions, which also yields the global uniqueness of the inverse conductivity problem \cite{SU1987}. In two dimensions, recent work by Bukhgeim \cite{B2007} demonstrated uniqueness, at any fixed $k$, by introducing stationary phase methods. At $k = 0$, a logarithmic stability estimate of recovering the potential function is given by Alessandrini \cite{Alessandrini1988} and is further proven to be optimal by Mandache \cite{M2001}. As for the numerical reconstruction of the potential function (or the conductivity function) at a fixed wavenumber, we refer to \cite{DS1994} and recent work \cite{HLSST2016}. We shall mention that the inverse medium problem in \cite{BL2005} also highly relates with the current work.

Recently it has been widely observed in different inverse boundary value problems that the stability estimate improves with growing wavenumbers (or energy) both analytically and numerically. The increasing stability in the inverse boundary value problem for the Schr\"{o}dinger equation (\ref{eqn:mainprob}) was firstly observed within different ranges of the wavenumbers in \cite{I2011}. In \cite{IW2014}, the authors considered the increasing stability on inverse boundary value problems for the Schr\"{o}dinger equation with attenuation where a linear exponential dependence on the attenuation constant is established. We note that all the above mentioned increasing stability results on inverse Schr\"{o}dinger potential problems are considered in three and higher dimensions under different a priori regularity assumptions.

Another multifrequency set-up uses single observation for multiple wavenumbers. To have an overview of such results, we recommend a recent review paper \cite{BLLT2015} which nicely summarizes the theoretical and numerical evidences verifying the increasing stability in inverse medium and source problems for acoustic Helmholtz equations and Maxwell equations.

In this paper, we investigate a linearized form of the inverse Schr\"{o}dinger potential problem (\ref{eqn:mainprob}) with a large wavenumber in two and higher dimensions. Such a linearization is carried out at a zero potential function which is reasonable if $c$ is sufficiently small compared with the squared wavenumber $k^{2}$. Noticing that the squared wavenumber $k^{2}$ may be large, we are allowed to reconstruct a potential function of moderate amplitude.
In Section \ref{se_sta} we introduce the linearized problem and obtain an (increasing) H\"{o}lder-type stability estimate for the linearized inverse Schr\"{o}dinger potential problem with a large wavenumber by using bounded complex exponential solutions. We extend, in Section \ref{se_atten}, the discussion to the inverse Schr\"{o}dinger potential problem with attenuation. An exponential dependence on the attenuation constant in the stability estimate is observed. A novel reconstruction algorithm is proposed to recover the Fourier modes of the potential function in Section \ref{se_algo} based on the Calder\'{o}n's method. By choosing the large wavenumber appropriately, we show various numerical examples confirming the efficiency of the proposed algorithm in Section \ref{se_numer}. Finally a conclusion Section \ref{se_con} ends the manuscript with further prospects.


\section{Increasing stability in the linearized inverse Schr\"{o}dinger potential problem at the large wavenumber}\label{se_sta}

We recall that the original problem, which initializes the current work, is to find the Schr\"{o}dinger potential function $c = c(x)$ defined in a bounded domain $\Omega$ in the following problem
\begin{equation}\label{eqn:prob}
(I) \quad \left\{
\begin{aligned}
- \Delta u - (k^{2} - c) u &= 0 & & \text{in\ } \Omega \subset \mathbb{R}^{n},\\
u &= g_{0} & & \text{on\ } \partial \Omega,
\end{aligned}
\right.
\end{equation}
from the knowledge of the Dirichlet-to-Neumann (DtN) map
\begin{equation}\label{eqn:dtn}
\Lambda: g_{0} \mapsto g_{1} := \partial_{\nu} u \quad \text{on\ } \partial \Omega.
\end{equation}
We assume that $k^2$ is not a Dirichlet eigenvalue for the Laplace operator in $\Omega$.

If we assume that $c$ is small (or $k$ is large), we can justify the linearization of the Schr\"{o}dinger equation. More precisely, let $u_{0}$, $u_{1}$ solve the following sub-problems
\begin{subequations}
\begin{align}
(I_{0}) \quad & \left\{
\begin{aligned}
- \Delta u_{0} - k^{2} u_{0} &= 0 \quad & & \text{in\ } \Omega,\\
u_{0} &= g_{0} & & \text{on\ } \partial \Omega,
\end{aligned}
\right. \label{eqn:sub_I_0}\\
(I_{1}) \quad & \left\{
\begin{aligned}
- \Delta u_{1} - k^{2} u_{1} &= - c u_{0} & & \text{in\ } \Omega,\\
u_{1} &= 0 & & \text{on\ } \partial \Omega,
\end{aligned}
\right. \label{eqn:sub_I_1}
\end{align}
\end{subequations}
then the solution $u$ of the original problem (\ref{eqn:prob}) is
\begin{equation*}
u = u_{0} + u_{1} + \cdots
\end{equation*}
where the remaining "$\cdots$" are "higher" order terms. The linearzied DtN map $\Lambda^{\prime}$ of $\Lambda$ in (\ref{eqn:dtn}) is defined accordingly as
\begin{equation}\label{eqn:dtn_p}
\Lambda^{\prime} : g_{0} \mapsto \partial_{\nu} u_{1} \quad \text{on\ } \partial \Omega.
\end{equation}

Multiplying both sides of the sub-problem (\ref{eqn:sub_I_1}) with a test function $v \in H^{1}(\Omega)$ satisfying the equation
\begin{equation*}
- \Delta v - k^{2} v = 0 \quad \text{in\ } \Omega,
\end{equation*}
we obtain
\begin{equation}\label{eqn:integral}
\int_{\Omega} c \, u_{0} v = \int_{\partial \Omega} \left( \partial_{\nu} u_{1} \right) v.
\end{equation}

Without loss of generality we may assume that $0 \in \Omega$. Let $D = 2 \sup\limits_{x \in \Omega} \left| x \right|$ and $\epsilon$ be the operator norm of $\Lambda^{\prime} : H^{\frac{1}{2}}(\partial \Omega) \to H^{-\frac{1}{2}}(\partial \Omega)$.
We present the main stability estimate below:
\begin{thm}\label{thm_1}
Let $D \leqslant 1$, $\left\| c \right\|_{H^{1}(\Omega)} \leqslant M_{1}$, and $k > 1$, $\epsilon < 1$, then the following estimate holds true
\begin{equation*}
\left\| c \right\|_{L^{2}(\Omega)}^{2}
\leqslant C(\Omega) \left( k^{n+4} + E^{n+4} \right) \epsilon^{2}
+ C(\Omega) E^{n+2} \left( \epsilon + \epsilon^{3} \right)
+ \frac{M_{1}^{2}}{1 + E^{2} + 3k^{2}}
\end{equation*}
for the linearized system (\ref{eqn:sub_I_0})-(\ref{eqn:sub_I_1}) with $E = -\ln \epsilon$ and the constant $C(\Omega)$ depending on the domain $\Omega$.
\end{thm}

\begin{proof}
The proof is based on the complex exponential solutions suggested by Calder\'{o}n \cite{C1980} and Faddeev \cite{F1966}. Let $\xi \in \mathbb{R}^{n}$ and $\zeta, \zeta^{*} \in \mathbb{C}^{n}$ with $\xi,  \zeta, \zeta^{*} \neq 0$. We consider the following solutions
\begin{equation}\label{eqn:uvcgo}
u_{0}(x) = e^{\mathbf{i} \zeta \cdot x}, \quad
v(x) = e^{\mathbf{i} \zeta^{*} \cdot x}.
\end{equation}
To select $\zeta$ and $\zeta^{*}$, we introduce an orthonormal base $\left\{ e_{1} := \frac{\xi}{\left| \xi \right|}, e_{2}, \cdots, e_{n} \right\}$ of $\mathbb{R}^{n}$ and let
\begin{equation*}
\zeta := \frac{\left| \xi \right|}{2} e_{1} + \sqrt{k^{2} - \frac{\left| \xi \right|^{2}}{4}} e_{2}, \quad
\zeta^{*} := \frac{\left| \xi \right|}{2} e_{1} - \sqrt{k^{2} - \frac{\left| \xi \right|^{2}}{4}} e_{2}
\end{equation*}
where $\sqrt{k^{2} - \frac{\left| \xi \right|^{2}}{4}} = \mathbf{i} \sqrt{\frac{\left| \xi \right|^{2}}{4} - k^{2}}$ if $k < \frac{\left| \xi \right|}{2}$. For brevity, we denote $\Xi := \sqrt{\frac{\left| \xi \right|^{2}}{4} - k^{2}}$.
Then $u_{0}(x) v(x) = e^{\mathbf{i} \xi \cdot x}$ and from the equality (\ref{eqn:integral})
we derive
\begin{equation}\label{eqn:integral_F}
\mathcal{F}[c](\xi)
:= \int_{\Omega} c(x) e^{\mathbf{i}\xi \cdot x} \,\mathrm{d}x
= \int_{\partial \Omega} \left( \partial_{\nu} u_{1} \right) v.
\end{equation}
Here $\mathcal{F}[\cdot]$ denotes the Fourier transform.

Observe that if $k \geqslant \frac{\left| \xi \right|}{2}$, then the norms of exponential solutions in (\ref{eqn:uvcgo}) are bounded as follow
\begin{equation}\label{eqn:xi_1}
\left\| u_{0} \right\|_{H^{1}(\Omega)}^{2}
= \left\| v \right\|_{H^{1}(\Omega)}^{2}
= \int_{\Omega} \left( 1 + \left| \zeta \right|^{2} \right) \,\mathrm{d}x
\leqslant \left( 1 + k^{2} \right) \operatorname{Vol}_{n} \Omega,
\end{equation}
where $\operatorname{Vol}_{n} \Omega$ is the volume of $\Omega \subset \mathbb{R}^{n}$. If $k < \frac{\left| \xi \right|}{2}$, then
\begin{equation}\label{eqn:xi_2}
\begin{aligned}
\left\| u_{0} \right\|_{H^{1}(\Omega)}^{2}
= \left\| v \right\|_{H^{1}(\Omega)}^{2}
&= \left( 1 + k^{2} \right) \int_{\Omega} e^{- 2 \Xi \, e_{2} \cdot x} \,\mathrm{d}x \\
&\leqslant \left( 1 + k^{2} \right) \operatorname{Vol}_{n-1} \Omega \,\frac{e^{D \Xi} - e^{-D \Xi}}{2 \Xi},
\end{aligned}
\end{equation}
where $\operatorname{Vol}_{n-1} \Omega = \sup\limits_{\Omega^{\prime}} \{ \operatorname{Vol}_{n-1} \Omega^{\prime} \}$ over all $(n-1)$-dimensional orthonormal projections $\Omega^{\prime}$ of $\Omega$.

By the trace theorem, there hold
\begin{equation}\label{eqn:trace}
\left\| u_{0} \right\|_{H^{\frac{1}{2}}(\partial \Omega)}
\leqslant C(\Omega) \left\| u_{0} \right\|_{H^{1}(\Omega)}, \quad
\left\| v \right\|_{H^{\frac{1}{2}}(\partial \Omega)}
\leqslant C(\Omega) \left\| v \right\|_{H^{1}(\Omega)}.
\end{equation}
So from (\ref{eqn:integral_F}) and (\ref{eqn:trace}), we have
\begin{equation*}
\begin{aligned}
\left| \mathcal{F}[c](\xi) \right|^{2}
&\leqslant \left\| \partial_{\nu} u_{1} \right\|_{H^{-\frac{1}{2}}(\partial \Omega)}^{2} \left\| v \right\|_{H^{\frac{1}{2}}(\partial \Omega)}^{2} \\
&\leqslant \epsilon^{2} \left\| u_{0} \right\|_{H^{\frac{1}{2}}(\partial \Omega)}^{2} \left\| v \right\|_{H^{\frac{1}{2}}(\partial \Omega)}^{2} \\
&\leqslant \epsilon^{2} C^{4}(\Omega) \left\| u_{0} \right\|_{H^{1}(\Omega)}^{2} \left\| v \right\|_{H^{1}(\Omega)}^{2}.
\end{aligned}
\end{equation*}
Hence, due to (\ref{eqn:xi_1}) and (\ref{eqn:xi_2}),
\begin{equation*}
\left| \mathcal{F}[c](\xi) \right|^{2}
\leqslant C^{4}(\Omega) \left( 1 + k^{2} \right)^{2} \left( \operatorname{Vol}_{n} \Omega \right)^{2} \epsilon^{2},
\quad \text{if\ } k \geqslant \frac{\left| \xi \right|}{2},
\end{equation*}
and
\begin{equation}\label{eq_nminus1bound}
\left| \mathcal{F}[c](\xi) \right|^{2}
\leqslant C^{4}(\Omega) \left( 1 + k^{2} \right)^{2} \left( \operatorname{Vol}_{n-1} \Omega \right)^{2} \frac{\left( e^{D \Xi} - e^{-D \Xi} \right)^2}{4 \Xi^{2}} \epsilon^{2},
\quad \text{if\ } k < \frac{\left| \xi \right|}{2}.
\end{equation}

Let $E := - \ln \epsilon > 0$ and $k > 1$, $\epsilon <1$, we consider two cases
\begin{enumerate}
\item[a)] $k > E$ (i.e. $\epsilon = e^{-E} > e^{-k}$), and
\item[b)] $k \leqslant E$ (i.e. $\epsilon = e^{-E} \leqslant e^{-k}$).
\end{enumerate}

In the case a), we have
\begin{equation}\label{eqn:bound_a}
\begin{aligned}
\left\| c \right\|_{L^{2}(\Omega)}^{2}
&= \int \left| \mathcal{F}[c](\xi) \right|^{2} \,\mathrm{d}\xi
= \int_{k \geqslant \frac{\left| \xi \right|}{2}} \left| \mathcal{F}[c](\xi) \right|^{2} \,\mathrm{d}\xi
+ \int_{k < \frac{\left| \xi \right|}{2}} \left| \mathcal{F}[c](\xi) \right|^{2} \,\mathrm{d}\xi \\
&\leqslant C^{4}(\Omega) \left( 1 + k^{2} \right)^{2} \left( \operatorname{Vol}_{n} \Omega \right)^{2} \sigma_{n} (2k)^{n} \epsilon^{2}
+ \frac{M_{1}^{2}}{1 + (2k)^{2}} \\
&\leqslant C_{1}(\Omega) k^{n+4} \epsilon^{2} + \frac{M_{1}^{2}}{1 + 4 k^{2}}
\leqslant C_{1}(\Omega) k^{n+4} \epsilon^{2} + \frac{M_{1}^{2}}{1 + E^{2} + 3 k^{2}}
\end{aligned}
\end{equation}
where $\sigma_{n}$ is the volume of an unit ball in $\mathbb{R}^{n}$, and the constant $C_{1}(\Omega)$ is defined as $C^{4}(\Omega) \left( \operatorname{Vol}_{n} \Omega \right)^{2} \sigma_{n} 2^{n+2}$.

In the case b), we let $\rho^{2} := \frac{E^{2}}{D^{2}} + 4 k^{2}$ and split
\begin{equation*}
\begin{aligned}
\left\| c \right\|_{L^{2}(\Omega)}^{2}
&= \int_{k \geqslant \frac{\left| \xi \right|}{2}} \left| \mathcal{F}[c](\xi) \right|^{2} \,\mathrm{d}\xi
+ \int_{k < \frac{\left| \xi \right|}{2} < \frac{\rho}{2}} \left| \mathcal{F}[c](\xi) \right|^{2} \,\mathrm{d}\xi
+ \int_{\rho \leqslant \left| \xi \right|} \left| \mathcal{F}[c](\xi) \right|^{2} \,\mathrm{d}\xi \\
&\leqslant C^{4}(\Omega) \left( 1 + k^{2} \right)^{2} \left( \operatorname{Vol}_{n} \Omega \right)^{2} \sigma_{n} (2k)^{n} \epsilon^{2} \\
&\quad + C^{4}(\Omega) \left( 1 + k^{2} \right)^{2} \left( \operatorname{Vol}_{n-1} \Omega \right)^{2} \left( \int_{k < \frac{\left| \xi \right|}{2} < \frac{\rho}{2}} \,\mathrm{d}\xi \right) \frac{D^{2} \left( e^{\frac{E}{2}} - e^{-\frac{E}{2}} \right)^{2}}{E^{2}} \epsilon^{2} \\
&\quad + \frac{M_{1}^{2}}{1 + \rho^{2}}.
\end{aligned}
\end{equation*}
The second term in the above inequality is derived by (\ref{eq_nminus1bound}) and the fact that
\begin{equation*}
\frac{e^{y}-e^{-y}}{y}
= 2 \left( 1 + \frac{y^{2}}{3!} + \cdots + \frac{y^{2n}}{(2n+1)!} + \cdots \right)
\end{equation*}
increases while $y > 0$ and hence
\begin{equation*}
\frac{e^{D \Xi} - e^{-D \Xi}}{2 \Xi}
\leqslant \frac{D\left( e^{\frac{E}{2}} - e^{-\frac{E}{2}} \right)}{E},
\quad {\rm since} \quad  k < \frac{\left| \xi \right|}{2} < \frac{\rho}{2} = \sqrt{\frac{E^{2}}{4D^{2}} + k^{2}}.
\end{equation*}
Observe that, since $k \leqslant E$,
\begin{equation*}
\begin{aligned}
\int_{k < \frac{\left| \xi \right|}{2} < \frac{\rho}{2}} \,\mathrm{d}\xi
&= \sigma_{n} \left[ \rho^{n} - (2k)^{n} \right]
= \sigma_{n} \left[ \left( \frac{E^{2}}{D^{2}} + (2k)^{2} \right)^{\frac{n}{2}} - (2k)^{n} \right] \\
&= \sigma_{n} \frac{E^{n}}{D^{n}} \left[ \left( 1 + \left( 2 k \frac{D}{E} \right)^{2} \right)^{\frac{n}{2}} - \left( 2 k \frac{D}{E} \right)^{n} \right] \\
&\leqslant \sigma_{n} \frac{E^{n}}{D^{n}} \left[ \left( 1 + \left( 2 D \right)^{2} \right)^{\frac{n}{2}} - \left( 2 D \right)^{n} \right].
\end{aligned}
\end{equation*}
Thus
\begin{equation}\label{eqn:bound_b}
\begin{aligned}
\left\| c \right\|_{L^{2}(\Omega)}^{2}
&\leqslant C^{4}(\Omega) \left( \operatorname{Vol}_{n} \Omega \right)^{2} \sigma_{n} 2^{n+2} k^{n+4} \epsilon^{2} \\
&\quad + C^{4}(\Omega) \left( \operatorname{Vol}_{n-1} \Omega \right)^{2} \sigma_{n} 2^{2} k^{4} \frac{E^{n-2}}{D^{n-2}} \left[ \left( 1 + \left( 2 D \right)^{2} \right)^{\frac{n}{2}} - \left( 2 D \right)^{n} \right] \\
&\quad \qquad \cdot \left( e^{E} + e^{-E} - 2 \right) \epsilon^{2} \\
&\quad + \frac{M_{1}^{2}}{1 + \frac{E^{2}}{D^{2}} + 4 k^{2}} \\
&\leqslant C_{1}(\Omega) E^{n+4} \epsilon^{2}
+ C_{2}(\Omega) E^{n+2} \left( \epsilon + \epsilon^{3} \right)
+ \frac{M_{1}^{2}}{1 + \frac{E^{2}}{D^{2}} + 4 k^{2}}
\end{aligned}
\end{equation}
where $C_{2}(\Omega) := C^{4}(\Omega) \left( \operatorname{Vol}_{n-1} \Omega \right)^{2} \sigma_{n} \frac{2^{2}}{D^{n-2}} \left[ \left( 1 + \left( 2 D \right)^{2} \right)^{\frac{n}{2}} - \left( 2 D \right)^{n} \right]$.

Combining case a) $k > E$ with the bound (\ref{eqn:bound_a}) and case b) $k \leqslant E$ with the bound (\ref{eqn:bound_b}), we prove Theorem \ref{thm_1}.
\end{proof}

\begin{rmk}\label{rem_1}
When $\Omega$ is a sphere, constants in the trace theorem are explicit and one can make $C(\Omega)$ in Theorem \ref{thm_1} explicitly.
\end{rmk}

In view of (\ref{eqn:bound_b}) it is obvious that stability is increasing with growing wavenumber $k$, provided $k \leqslant E$. We comment on the increasing stability when $k > E$. In such a case, we have an upper bound estimate (\ref{eqn:bound_a}) of $c$ which is
\begin{equation*}
\begin{aligned}
\left\| c \right\|_{L^{2}(\Omega)}^{2}
\leqslant \omega(k)
:= C_{1}(\Omega) k^{n+4} \epsilon^{2} + \frac{M_{1}^{2}}{4 k^{2}}.
\end{aligned}
\end{equation*}
By an elementary calculus, the upper bound $\omega(k)$ decreases on $(0,k_{*})$ and increases on $(k_{*},\infty)$ where
\begin{equation}\label{eq_kstar}
\begin{aligned}
k_{*} = \left( \frac{M_{1}^{2}}{2 (n+4) C_{1}(\Omega) \epsilon^{2}} \right)^{\frac{1}{n+6}}.
\end{aligned}
\end{equation}

If $k_{*} \leqslant 1$ in (\ref{eq_kstar}), we have
\begin{equation*}
\begin{aligned}
M_{1}^{2} \leqslant 2 (n+4) C_{1}(\Omega) \epsilon^{2}
\end{aligned}
\end{equation*}
and the minimum of $\omega(k)$ is obtained at $k = 1$ such that
\begin{equation*}
\begin{aligned}
\left\| c \right\|_{L^{2}(\Omega)}^{2}
\leqslant \omega(1) = C_{1}(\Omega) \epsilon^{2} + \frac{M_{1}^{2}}{4}
\leqslant C_{1}(\Omega) \left( 3 + \frac{n}{2} \right) \epsilon^{2}.
\end{aligned}
\end{equation*}
If $k_{*} > 1$, then minimum of $\omega(k)$ is obtained at $k = k_{*}$ such that
\begin{equation*}
\begin{aligned}
\left\| c \right\|_{L^{2}(\Omega)}^{2}
\leqslant \omega(k_{*})
= C_{1}^{\frac{2}{n+6}}(\Omega) M_{1}^{2\frac{n+4}{n+6}} \left[ \left( 2(n+4) \right)^{-\frac{n+4}{n+6}} + 2^{-2} \left( 2(n+4) \right)^{\frac{2}{n+6}} \right] \epsilon^{\frac{4}{n+6}}.
\end{aligned}
\end{equation*}

Now we consider the "small" perturbation when $E \geqslant 1$ (i.e., $\epsilon \leqslant e^{-1}$). If $k_{*} \leqslant E$, then minimum of $\omega(k)$ (at $k = E$) is
\begin{equation*}
\begin{aligned}
\left\| c \right\|_{L^{2}(\Omega)}^{2}
& \leqslant \omega(E) = C_{1}(\Omega) E^{n+4} \epsilon^{2} + \frac{M_{1}^2}{4 E^{2}} \\
& \leqslant C_{1}(\Omega) E^{n+4} \epsilon^{2} + C_{1}^{\frac{2}{n+6}}(\Omega) M_{1}^{2\frac{n+4}{n+6}} 2^{-2} \left( 2(n+4) \right)^{\frac{2}{n+6}} \epsilon^{\frac{4}{n+6}}.
\end{aligned}
\end{equation*}
If $k_{*} > E$, then minimum of $\omega(k)$ is at $k = k_{*}$ which is the same as in the case $0 < E < 1$. So in all possible cases above, given a small error (or a perturbation) $\epsilon$ we can choose a large wavenumber $k$ producing a H\"{o}lder stability which is a big improvement over a logarithmic stability at low wavenumbers.


\section{Linearized inverse Schr\"{o}dinger potential problem at the large wavenumber with attenuation}\label{se_atten}

Furthermore we provide some extended discussion on the linearized inverse Schr\"{o}dinger potential problem with attenuation. The original problem in three (and higher) dimensions is studied in \cite{IW2014} where increasing stability estimates with a linearly exponential dependence of the attenuation constant are obtained for low/high frequencies separately by constructing almost complex exponential solutions and real solutions respectively.

In current section, we investigate the stability estimate recovering the potential function $c = c(x)$ below
\begin{equation}\label{eqn:mainprob_attenuation}
\left\{
\begin{aligned}
- \Delta u - (k^{2} - c) u + \mathbf{i}kb u &= 0 & & \text{in\ } \Omega \subset \mathbb{R}^{n},\\
u &= g_{0} & & \text{on\ } \partial \Omega,
\end{aligned}
\right.
\end{equation}
from the knowledge of linearized DtN map for a large wavenumber $k$. The constant $b > 0$ is the attenuation constant. For the sake of simplicity, we use the notations in Section \ref{se_sta}, for instance, that $u$ represents the solution of the Schr\"{o}dinger potential equation with attenuation.

Referring to Section \ref{se_sta}, we let $u_{0}$, $u_{1}$ solve the following sub-problems
\begin{subequations}
\begin{align}
& \left\{
\begin{aligned}
- \Delta u_{0} - k^{2} u_{0} + \mathbf{i}kb \, u_{0} &= 0 \quad & & \text{in\ } \Omega,\\
u_{0} &= g_{0} & & \text{on\ } \partial \Omega,
\end{aligned}
\right. \label{eqn:sub_I_0_attenuation}\\
& \left\{
\begin{aligned}
- \Delta u_{1} - k^{2} u_{1} + \mathbf{i}kb \, u_{1} &= - c u_{0} & & \text{in\ } \Omega,\\
u_{1} &= 0 & & \text{on\ } \partial \Omega,
\end{aligned}
\right. \label{eqn:sub_I_1_attenuation}
\end{align}
\end{subequations}
and the solution $u$ of (\ref{eqn:mainprob_attenuation}) is
\begin{equation*}
u = u_{0} + u_{1} + \cdots
\end{equation*}
with "$\cdots$" denoting the remaining terms. Similarly the linearized DtN map $\Lambda^{\prime}$ is defined as
\begin{equation}\label{eqn:dtn_atten}
\Lambda^{\prime}: g_{0} \mapsto \partial_{\nu} u_{1} \quad {\rm on} \quad \partial \Omega.
\end{equation}

Multiplying both sides of the sub-problem (\ref{eqn:sub_I_1_attenuation}) with a test function $v$ satisfying
\begin{equation*}
-\Delta v - k^{2} v + \mathbf{i}kb \, v = 0 \quad {\rm in }\quad \Omega,
\end{equation*}
we obtain the equality
\begin{equation}\label{eqn:integral_atten}
\int_{\Omega} c \, u_{0} v = \int_{\partial\Omega} \left( \partial_{\nu} u_{1} \right) v
\end{equation}
which will again play important roles below.

Let $\epsilon$ be the operator norm of $\Lambda^{\prime} : H^{\frac{1}{2}}(\partial \Omega) \to H^{-\frac{1}{2}}(\partial \Omega)$, $0\in \Omega$ and $D := 2 \sup \left| x \right|$, $x\in \Omega$. We present the main stability estimate in this section.

\begin{thm}\label{thm_2}
Let $D \leqslant 1$, $\left\| c \right\|_{H^{1}(\Omega)} \leqslant M_{1}$, and $k > 1$, $\epsilon < 1$, then the following estimate holds true
\begin{equation*}
\begin{aligned}
\left\| c \right\|_{L^{2}(\Omega)}^{2}
&\leqslant C(\Omega) \left( k^{n+4} + E^{n+4} \right) e^{2D b} \epsilon^{2} \\
&\quad + C(\Omega) E^{n+4} e^{D^{2} b} \epsilon + C(\Omega) E^{n+2} e^{D^{2} b} \epsilon + \frac{M_{1}^{2}}{1 + E^{2} + 2 k^{2}}
\end{aligned}
\end{equation*}
for the linearized system (\ref{eqn:sub_I_0_attenuation})-(\ref{eqn:sub_I_1_attenuation})
with $E = -\ln \epsilon$ and the constant $C(\Omega)$ depending on the domain $\Omega$.
\end{thm}

\begin{proof}
Similar to the proof of Theorem \ref{thm_1}, we again use the exponential solutions.

Let $\xi \in \mathbb{R}^{n}$ and $\zeta, \zeta^{*} \in \mathbb{C}^{n}$ with $\xi, \zeta, \zeta^{*} \neq 0$. We consider the following solution
\begin{equation*}
u_{0}(x) = e^{\mathbf{i} \zeta \cdot x},
\quad v(x) = e^{\mathbf{i} \zeta^{*} \cdot x}.
\end{equation*}
The orthonormal base is chosen by $\left\{ e_{1} = \frac{\xi}{\left| \xi \right|}, e_{2}, \cdots, e_{n} \right\}$ and we let
\begin{equation*}
\zeta := \frac{\left| \xi \right|}{2} e_{1} + \sqrt{ k^{2} - \frac{\left| \xi \right|^{2}}{4} - \mathbf{i}kb } \, e_{2},
\quad \zeta^{*} := \frac{\left| \xi \right|}{2} e_{1} - \sqrt{ k^{2} - \frac{\left| \xi \right|^{2}}{4} - \mathbf{i}kb } \, e_{2}.
\end{equation*}
Then $u_{0}(x) v(x) = e^{\mathbf{i} \xi \cdot x}$ and we derive, by (\ref{eqn:integral_atten})
\begin{equation*}
\mathcal{F}[c](\xi)
:= \int_{\Omega} c(x) e^{\mathbf{i} \xi \cdot x} \,\mathrm{d}x
= \int_{\partial\Omega} \left( \partial_{\nu} u_{1} \right) v.
\end{equation*}

We can write $\sqrt{ k^{2} - \frac{\left| \xi \right|^{2}}{4} - \mathbf{i}kb } = X + \mathbf{i} Y$, $X > 0$.
If $\left| \xi \right|^{2} \leqslant 3 k^{2}$, then squaring the both sides and following the arguments in \cite{IW2014}, we have
\begin{equation*}
\left| Y \right| = \frac{kb}{2X}
= \frac{kb}{\sqrt{2} \sqrt{\left( k^{2} - \frac{\left| \xi \right|^{2}}{4} \right) + \sqrt{\left( k^{2} - \frac{\left| \xi \right|^{2}}{4} \right)^{2} + k^{2} b^{2}}}}
\leqslant  b.
\end{equation*}
If $3 k^{2} < \left| \xi \right|^{2} \leqslant 4 k^{2}$, we derive
\begin{equation*}
\left| Y \right| =  \frac{kb}{\sqrt{2} \sqrt{\left( k^{2} - \frac{\left| \xi \right|^{2}}{4} \right) + \sqrt{\left( k^{2} - \frac{\left| \xi \right|^{2}}{4} \right)^{2} + k^{2} b^{2}}}}
\leqslant \frac{1}{2} \left( \frac{k}{D} + D b \right).
\end{equation*}
On the other hand, if $\left| \xi \right|^{2} > 4 k^{2}$, then
\begin{equation}\label{eqn:se5_Y}
Y = -\frac{kb}{2 X}
= -\frac{1}{2} \sqrt{\left( \frac{\left| \xi \right|^{2}}{2} - 2 k^{2} \right) + \sqrt{\left( \frac{\left| \xi \right|^{2}}{2} - 2 k^{2} \right)^{2} + 4 k^{2} b^{2}}}.
\end{equation}

We then derive that
\begin{itemize}

\item if $\left| \xi \right|^{2} \leqslant 3 k^{2}$, then
\begin{equation*}
\begin{aligned}
\left\| u_{0} \right\|_{H^{1}(\Omega)}^{2}
= \left\| v \right\|_{H^{1}(\Omega)}^{2}
& \leqslant \left( 1 + k^{2} \right) \int_{\Omega} e^{-2 Y e_{2} \cdot x} \,\mathrm{d}x \\
& \leqslant \left( 1 + k^{2} \right) \operatorname{Vol}_{n} \Omega \, e^{D b}.
\end{aligned}
\end{equation*}

\item If $3 k^{2} < \left| \xi \right|^{2} \leqslant 4 k^{2}$, then
\begin{equation*}
\begin{aligned}
\left\| u_{0} \right\|_{H^{1}(\Omega)}^{2}
= \left\| v \right\|_{H^{1}(\Omega)}^{2}
& \leqslant \left( 1 + k^{2} \right) \int_{\Omega} e^{-2 Y e_{2} \cdot x} \,\mathrm{d}x \\
& \leqslant \left( 1 + k^{2} \right) \operatorname{Vol}_{n} \Omega \, e^{\frac{1}{2} ( k + D^{2} b )}.
\end{aligned}
\end{equation*}

\item If $\left| \xi \right|^{2} > 4 k^{2}$, there holds
\begin{equation*}
\begin{aligned}
\left\| u_{0} \right\|_{H^{1}(\Omega)}^{2}
= \left\| v \right\|_{H^{1}(\Omega)}^{2}
& \leqslant \left( 1 + k^{2} \right) \operatorname{Vol}_{n-1} \Omega \int_{-\frac{D}{2}}^{\frac{D}{2}}  e^{-2 Y t} \,\mathrm{d}t \\
& = \left( 1 + k^{2} \right) \operatorname{Vol}_{n-1} \Omega \, \frac{ e^{Dy} - e^{-Dy} }{2y}
\end{aligned}
\end{equation*}
where a variable $y := -Y > 0$ due to (\ref{eqn:se5_Y}).
\end{itemize}
Similarly to the proof of Theorem \ref{thm_1}, we have
\begin{equation*}
\begin{aligned}
\left| \mathcal{F}[c](\xi) \right|^{2}
& \leqslant \epsilon^{2} C^{4}(\Omega) \left\| u_{0} \right\|_{H^{1}(\Omega)}^{2} \left\| v \right\|_{H^{1}(\Omega)}^{2},
\end{aligned}
\end{equation*}
which further yields that
\begin{itemize}

\item if $\left| \xi \right|^{2} \leqslant 3 k^{2}$, then
\begin{equation*}
\begin{aligned}
\left| \mathcal{F}[c](\xi) \right|^{2}
\leqslant C^{4}(\Omega) \left( 1 + k^{2} \right)^{2} \left( \operatorname{Vol}_{n} \Omega \right)^{2} e^{2D b} \epsilon^{2}.
\end{aligned}
\end{equation*}

\item If $3 k^{2} < \left| \xi \right|^{2} \leqslant 4 k^{2}$, then
\begin{equation*}
\begin{aligned}
\left| \mathcal{F}[c](\xi) \right|^{2}
\leqslant C^{4}(\Omega) \left( 1 + k^{2} \right)^{2} \left( \operatorname{Vol}_{n} \Omega \right)^{2} e^{ k + D^{2} b } \epsilon^{2}.
\end{aligned}
\end{equation*}

\item If $\left| \xi \right|^{2} > 4 k^{2}$, then
\begin{equation*}
\begin{aligned}
\left| \mathcal{F}[c](\xi) \right|^{2}
\leqslant C^{4}(\Omega) \left( 1 + k^{2} \right)^{2} \left( \operatorname{Vol}_{n-1} \Omega \right)^{2} \frac{ \left( e^{Dy} - e^{-Dy} \right)^{2} }{4 y^{2}} \epsilon^{2}.
\end{aligned}
\end{equation*}

\end{itemize}

Let $E = -\ln \epsilon > 0$, $k > 1$ and $\epsilon < 1$. We again consider the cases below
\begin{enumerate}
\item[a)] $k > E$ (i.e. $\epsilon = e^{-E} > e^{-k}$), and
\item[b)] $k \leqslant E$ (i.e. $\epsilon = e^{-E} \leqslant e^{-k}$).
\end{enumerate}

In the case a), we have
\begin{equation}\label{eqn:thm2proof_estimate1}
\begin{aligned}
\left\| c \right\|_{L^{2}(\Omega)}^{2}
& = \int_{\left| \xi \right|^{2} \leqslant 3 k^{2}} \left| \mathcal{F}[c](\xi) \right|^{2} \,\mathrm{d}\xi + \int_{ 3 k^{2} < \left| \xi \right|^{2} } \left|\mathcal{F}[c](\xi) \right|^{2} \,\mathrm{d}\xi \\
& \leqslant C^{4}(\Omega) \left( 1 + k^{2} \right)^{2} \left( \operatorname{Vol}_{n} \Omega \right)^{2} \sigma_{n} \left( \sqrt{3} k \right)^{n} e^{2D b} \epsilon^{2} + \frac{M_{1}^{2}}{1 + 3 k^{2}} \\
& \leqslant \frac{1}{4}C_{1}(\Omega) k^{n+4} e^{2D b} \epsilon^{2} + \frac{M_{1}^{2}}{1 + E^{2} + 2 k^{2}}.
\end{aligned}
\end{equation}

In the case b), we again choose $\rho^{2} := \frac{E^{2}}{D^{2}} + 4 k^{2}$ and split
\begin{equation*}
\begin{aligned}
\left\| c \right\|_{L^{2}(\Omega)}^{2}
&= \int_{\left| \xi \right|^{2} \leqslant 3 k^{2}} \left| \mathcal{F}[c](\xi) \right|^{2} \,\mathrm{d}\xi + \int_{ 3 k^{2} < \left| \xi \right|^{2} \leqslant 4 k^{2} } \left| \mathcal{F}[c](\xi) \right|^{2} \,\mathrm{d}\xi \\
& \quad + \int_{ 4 k^{2} < \left| \xi \right|^{2} \leqslant \rho^{2} } \left|\mathcal{F}[c](\xi) \right|^{2} \,\mathrm{d}\xi + \int_{ \rho^{2} < \left| \xi \right|^{2} } \left| \mathcal{F}[c](\xi) \right|^{2} \,\mathrm{d}\xi \\
&\leqslant \frac{1}{4} C_1(\Omega) k^{n+4} e^{2Db} \epsilon^2 + \frac{1}{4} C_1(\Omega) \left(2^n- 3^{\frac{n}{2}}\right) k^n e^{k+D^2 b} \epsilon^2 \\
& \quad + C^{4}(\Omega) \left( 1 + k^{2} \right)^{2} \left( \operatorname{Vol}_{n-1} \Omega \right)^{2} \left( \int_{ 4 k^{2} < \left| \xi \right|^{2} \leqslant \rho^{2} } \frac{ \left( e^{Dy} - e^{-Dy} \right)^{2}}{4 y^{2}} \,\mathrm{d}\xi \right)  \epsilon^{2} \\
&\quad + \frac{M_{1}^{2}}{1 + \rho^{2}}.
\end{aligned}
\end{equation*}
As in Theorem \ref{thm_1}, $\frac{ e^{Dy} - e^{-Dy} }{y}$ increases while $y > 0$ and thus it attains the maximum value if $\left| \xi \right| = \rho$, which yields
\begin{equation*}
\begin{aligned}
\frac{ \left( e^{Dy} - e^{-Dy} \right)^{2}}{4 y^{2}}
\leqslant \frac{ e^{D \sqrt{ \frac{1}{2} \frac{E^{2}}{D^{2}} + \sqrt{ \left( \frac{1}{2} \frac{E^{2}}{D^{2}} \right)^{2} + 4 k^{2} b^{2} } }} }{ \frac{1}{2} \frac{E^{2}}{D^{2}} + \sqrt{ \left( \frac{1}{2} \frac{E^{2}}{D^{2}} \right)^{2} + 4 k^{2} b^{2} } }
\quad \text{because} \quad 4 k^{2} < \left| \xi \right|^{2} \leqslant \rho^{2}.
\end{aligned}
\end{equation*}
Since
\begin{equation*}
\begin{aligned}
\frac{1}{2} \frac{E^{2}}{D^{2}} + \sqrt{ \left( \frac{1}{2} \frac{E^{2}}{D^{2}} \right)^{2} + 4 k^{2} b^{2}} \geqslant \frac{E^{2}}{D^{2}}
\end{aligned}
\end{equation*}
and, by using $k \leqslant E$
\begin{equation*}
\frac{1}{2} \frac{E^{2}}{D^{2}} + \sqrt{ \left( \frac{1}{2} \frac{E^{2}}{D^{2}} \right)^{2} + 4 k^{2} b^{2} }
\leqslant \frac{E^{2}}{D^{2}} + 2 k b
\leqslant \frac{E^{2}}{D^{2}} + 2 \frac{E}{D} D b
\leqslant \left( \frac{E}{D} + D b \right)^{2},
\end{equation*}
we then bound
\begin{equation*}
\begin{aligned}
\frac{ \left( e^{Dy} - e^{-Dy} \right)^{2} }{4 y^{2}}
\leqslant \frac{ e^{ E + D^{2} b } }{ \left( \frac{E}{D} \right)^{2} }.
\end{aligned}
\end{equation*}
Meanwhile, by using $k \leqslant E$ and the proof of Theorem \ref{thm_1}, we yield
\begin{equation*}
\begin{aligned}
& \int_{ 4 k^{2} < \left| \xi \right|^{2} \leqslant \rho^{2} } \,\mathrm{d}\xi
\leqslant \sigma_{n} \frac{E^{n}}{D^{n}} \left[ \left( 1 + \left( 2 D \right)^{2} \right)^{\frac{n}{2}} - \left( 2 D \right)^{n} \right].
\end{aligned}
\end{equation*}

Combining above inequalities and using $k \leqslant E$ again, we conclude that
\begin{equation}\label{eqn:thm2proof_estimate2}
\begin{aligned}
\left\| c \right\|_{L^{2}(\Omega)}^{2}
&\leqslant \frac{1}{4} C_1(\Omega) k^{n+4} e^{2Db} \epsilon^2 + \frac{1}{4} C_1(\Omega) \left(2^n- 3^{\frac{n}{2}}\right) k^n e^{k+D^2 b} \epsilon^2 \\
& \quad + C^{4}(\Omega) \left( \operatorname{Vol}_{n-1} \Omega \right)^{2} \sigma_{n} \frac{2^{2}}{D^{n-2}} \left[ \left( 1 + \left( 2 D \right)^{2} \right)^{\frac{n}{2}} - \left( 2 D \right)^{n} \right] E^{n+2} e^{D^{2} b} \epsilon \\
& \quad + \frac{M_{1}^{2}}{1 + \frac{E^{2}}{D^{2}} + 4 k^{2}}.
\end{aligned}
\end{equation}
We end the proof by
combining the bounds (\ref{eqn:thm2proof_estimate1}) for $k > E$ and (\ref{eqn:thm2proof_estimate2}) for $k \leqslant E$.
\end{proof}

Following the discussion in Section \ref{se_sta}, we again observe the increasing stability with growing wavenumbers $k$. At the same time, because of the attenuation term, such an error estimate is deteriorated by the attenuation constant $b$ with a linearly exponential growth. Such a dependance is better than that for the multifrequency inverse source problem where a quadratically exponential growth is observed comparing the increasing stability estimates in \cite{CIL2016, IL2018}.

\section{Reconstruction algorithm}\label{se_algo}

In this section, we propose a reconstruction algorithm for the linearized inverse Schr\"{o}dinger potential problem at a large wavenumber, observing that the linearized problems are (\ref{eqn:sub_I_0}) and (\ref{eqn:sub_I_1}). For the sake of simplicity, we fix the dimensionality $n = 2$.

The linearized problem (\ref{eqn:sub_I_0})-(\ref{eqn:sub_I_1}) represents a linearized DtN map $\Lambda^{\prime}: g_{0} \mapsto \partial_{\nu} u_{1}$ by varying $g_{0}$. In particular, we can choose the solutions $u_{0}$ to be the bounded complex exponential solutions (\ref{eqn:uvcgo}) for $n = 2$ such that
\begin{equation}\label{eqn:numer_u0}
u_{0} = e^{\mathbf{i} \zeta \cdot x}
\quad \text{with} \quad \zeta = \frac{\xi}{2} + \sqrt{k^{2} - \frac{\left| \xi \right|^{2}}{4}} \, \xi^{\perp}
\end{equation}
by varying different vectors $\xi$ in the phase space and $g_{0} = \left. u_{0} \right|_{\partial \Omega}$, where $\xi^{\perp}$ satisfies $\xi \cdot \xi^{\perp} = 0$ and $\left| \xi^{\perp} \right| = 1$. At the same time, substituting each $u_{0}$ above into the second problem (\ref{eqn:sub_I_1}), we obtain the solution $u_{1}$ and its Neumann trace $\partial_{\nu} u_{1} = \Lambda^{\prime} g_{0}$ which also varies with different choice of $\xi$. Similarly we can choose the test function $v$ to be the relevant exponential solutions
\begin{equation}\label{eqn:numer_v}
v = e^{\mathbf{i} \zeta^{*} \cdot x}
\quad \text{with\ } \quad \zeta^{*} = \frac{\xi}{2} - \sqrt{k^{2} - \frac{\left| \xi \right|^{2}}{4}} \, \xi^{\perp}.
\end{equation}

To realize the reconstruction algorithm numerically, we shall approximate the linearized DtN map $\Lambda^{\prime}$ (\ref{eqn:dtn_p}) accurately. In particular, we mention that the Neumann boundary data $\partial_\nu u_{1}$ depends on the unknown potential function $c$ referring to (\ref{eqn:sub_I_1}). Assuming that the potential function $c$ is comparably small with respect to the squared wavenumber $k^{2}$ and following the standard linearization approach (for instance, in \cite{DS1994}), we define a linearized Neumann boundary data
\begin{equation}\label{eqn:lineardtn}
g_{1}^{\,\prime} := g_{1} - \partial_{\nu} u_{0} = \partial_{\nu} u - \partial_{\nu} u_{0},
\end{equation}
where $g_{1}$ is the Neumann boundary data in (\ref{eqn:dtn}) of the original problem (\ref{eqn:prob}) and $\partial_{\nu} u_{0}$ is the Neumann boundary data of the sub-problem (\ref{eqn:sub_I_0}). We thus approximate the linearized Neumann boundary data numerically by $ \partial_{\nu} u_{1} \approx g_{1}^{\,\prime}$, where the "higher" order terms are ignored in (\ref{eqn:lineardtn}). If we substitute the above solutions (\ref{eqn:numer_u0}), (\ref{eqn:numer_v}) into the equality (\ref{eqn:integral}), the recovery of $c$ is equivalent to solving the following integral equations
\begin{equation}\label{eqn:integral_F_v}
\mathcal{F}[c](\xi)
= \int_{\Omega} c(x) \, e^{\mathbf{i} \xi \cdot x} \,\mathrm{d}x
= \int_{\partial \Omega} \left( \partial_{\nu} u_{1} \right) e^{\mathbf{i} \zeta^{*} \cdot x} \,\mathrm{d}s_{x}
\approx \int_{\partial \Omega} g_{1}^{\,\prime} \, e^{\mathbf{i} \zeta^{*} \cdot x} \,\mathrm{d}s_{x}
\end{equation}
where vectors $\xi$, $\xi^{\perp}$ vary in the phase space with a fixed wavenumber $k$. The left hand side of the above formula (\ref{eqn:integral_F_v}) is the Fourier coefficient $\mathcal{F}[c](\xi)$ of the potential function $c$ at a point $\xi$. Thus the reconstruction of $c$ can be achieved by varying vectors $\xi$, $\xi^{\perp}$ and the inverse Fourier transform.

\begin{rmk}
Notice that, in current section, the chosen orthonormal base of the phase space $\mathbb{R}^{2}$ is $\left\{ e_{1} = \frac{\xi}{\left| \xi \right|}, \, e_{2} = \xi^{\perp} \right\}$. Thus the formula of a vector $\zeta$ (or $\zeta^{*}$) is fixed by $k$. Noticing that $\left| \xi \right| > 2k$ and $\sqrt{k^{2} - \frac{\left| \xi \right|^{2}}{4}} = \mathbf{i} \sqrt{\frac{\left| \xi \right|^{2}}{4} - k^{2}}$, we write the exponential solution $u_{0}$ (or $v$) explicitly by
\begin{equation}\label{eqn:numer_u0_xi}
u_{0}(x) = e^{\mathbf{i} \zeta \cdot x}
= \left\{
\begin{aligned}
&\exp \left\{ \mathbf{i} \left( \frac{\xi}{2} + \sqrt{k^{2} - \frac{\left| \xi \right|^{2}}{4}} \, \xi^{\perp} \right) \cdot x \right\}, & \left| \xi \right| \leqslant 2k, \\
&\exp \left\{ \mathbf{i} \frac{\xi}{2} \cdot x \right\} \exp \left\{ - \sqrt{\frac{\left| \xi \right|^{2}}{4} - k^{2}} \, \xi^{\perp} \cdot x \right\}, & \left| \xi \right| > 2k.
\end{aligned}
\right.
\end{equation}
Indeed, $u_{0}$ is a plane wave traveling along its wave-vector $\zeta$ with a length equal to the wavenumber $k$ if $\left| \xi \right| \leqslant 2k$. Meanwhile, if $\left| \xi \right| > 2k$, the (high) oscillation of the exponential solution $u_{0}$ is observed in a direction $e_{1} = \frac{\xi}{\left| \xi \right|}$ with the spatial frequency $\frac{\left| \xi \right|}{2}$, and decays exponentially along the unit vector $e_{2} = \xi^{\perp}$. Then numerical calculation in (\ref{eqn:integral}) or (\ref{eqn:integral_F_v}) becomes unstable when $\left| \xi \right| > 2k$.
\end{rmk}


Now we describe a reconstruction algorithm based on discrete sets of length and angle of the vectors in phase space. We first choose a discrete and finite length set
\begin{equation*}
\{\kappa_{\ell}\}_{\ell=1}^{M} \subset (0, m k \,]
\quad \text{for any fixed\ } k.
\end{equation*}
Here we choose $2 \leqslant m \in \mathbb{N}_{+}$ and $m k$ is the maximum length of the vector $\xi$. Combining with two unit vector (or angle) sets
\begin{equation*}
\{\hat{y}_{s}\}_{s=1}^{N} \subset \mathbb{S}^{n-1}
\quad \text{and} \quad \{\hat{z}_{s}\}_{s=1}^{N} \subset \mathbb{S}^{n-1},
\end{equation*}
which satisfy that $\hat{y}_{s} \cdot \hat{z}_{s} = 0$, we denote
\begin{equation*}
\xi^{\langle \ell;s \rangle} = \kappa_{\ell} \, \hat{y}_{s},
\quad
\zeta^{\langle \ell;s \rangle} = \frac{\kappa_{\ell}}{2} \hat{y}_{s} + \sqrt{k^{2} - \frac{\kappa_{\ell}^{2}}{4}} \, \hat{z}_{s},
\quad
\zeta^{\langle \ell;s \rangle}_{*} = \frac{\kappa_{\ell}}{2} \hat{y}_{s} - \sqrt{k^{2} - \frac{\kappa_{\ell}^{2}}{4}} \, \hat{z}_{s},
\end{equation*}
which are the vectors (or points) chosen in the phase space, while $\ell = 1,2,\cdots,M$ and $s = 1,2,\cdots,N$. More precisely, the superscript notation $\cdot^{\langle \ell;s \rangle}$ will be referred to a vector $\xi^{\langle \ell;s \rangle}$ with the $\ell$th length $\kappa_{\ell}$ and the $s$th angle $\hat{y}_{s}$. Finally, for the inverse Fourier transform, a numerical quadrature rule can be constructed by a suitable choice of the weights $\sigma^{\langle \ell;s \rangle}$ according to these points $\xi^{\langle \ell;s \rangle}$.

We summarize our reconstruction algorithm below.
\vspace{10pt}
\hrule\hrule
\vspace{8pt}
{\parindent 0pt \bf Algorithm 1: Reconstruction Algorithm for the Linearized Schr\"{o}dinger Potential Problem}
\vspace{8pt}
\hrule
\vspace{8pt}
{\parindent 0pt \bf Input:} %
$\{\kappa_{\ell}\}_{\ell=1}^{M}$, %
$\{\hat{y}_{s}\}_{s=1}^{N}$, $\{\hat{z}_{s}\}_{s=1}^{N}$ and %
$\sigma^{\langle \ell;s \rangle}$; \\[5pt]%
{\parindent 0pt \bf Output:} %
Approximated Potential $c_{\rm inv}= c^{\langle M+1;1 \rangle}$. \\[-5pt]%
\begin{enumerate}
  \item[1:] $\,$ Set $c^{\langle 1;1 \rangle} := 0$; %
  \item[2:] $\,$ {\bf For} $\ell = 1,2,\cdots,M$ (length~updating) %
  \item[3:] $\,$ \quad {\bf For} $s = 1,2,\cdots,N$ (angle~updating) %
  \item[4:] $\,$ \quad \quad Choose $u_{0} := \exp \{ \mathbf{i} \zeta^{\langle \ell;s \rangle} \cdot x \}$ and $g_{0} := \left. u_{0} \right|_{\partial \Omega}$; %
  \item[5:] $\,$ \quad \quad Measure the Neumann boundary data $\partial_{\nu} u$ of the forward problem (\ref{eqn:prob})
  \item[] $\,$ \quad \qquad given the Dirichlet boundary data $g_{0}$; %
  \item[6:] $\,$ \quad \quad Calculate the approximated Neumann boundary data $g_{1}^{\,\prime} = \partial_{\nu} u - \partial_{\nu} u_{0}$
  \item[] $\,$ \quad \qquad on the boundary $\partial \Omega$; %
  \item[7:] $\,$ \quad \quad Choose $v:= \exp \{ \mathbf{i} \zeta^{\langle \ell;s \rangle}_{*} \cdot x \}$ and $w := \big[ u_{0} v \big]^{-1} = \exp \{ -\mathbf{i} \xi^{\langle \ell;s \rangle} \cdot x \}$; %
  \item[8:] $\,$ \quad \quad Compute $\mathcal{F}[c](\xi^{\langle \ell;s \rangle}) \approx \int_{\partial \Omega} g_{1}^{\,\prime} \, v \,\mathrm{d}s_{x}$; %
  \item[9:] $\,$ \quad \quad Update $c^{\langle \ell;s+1 \rangle} = c^{\langle \ell;s \rangle} + \sigma^{\langle \ell;s \rangle} \mathcal{F}[c](\xi^{\langle \ell;s \rangle}) \, w$; %
  \item[10:] $\,$ \quad {\bf End} %
  \item[11:] $\,$ \quad Set $c^{\langle \ell+1;1 \rangle} := c^{\langle \ell;N+1 \rangle}$; %
  \item[12:] $\,$ {\bf End}. %
\end{enumerate}
\vspace{8pt}
\hrule\hrule
\vspace{10pt}

The above Algorithm 1 provides us with a reconstructed potential function $c_{\rm inv}$ to mimic the exact one. The main error between these two potential functions consists of two parts. The first one is the approximation error because of the discrete and finite length set and angle sets. The second one is the numerical error of approximated linearized DtN map $\Lambda^{\prime}$ in (\ref{eqn:integral_F_v}) where the elliptic equation solvers of large wavenumbers and the numerical differentiation of the Neumann boundary data may induce some additional ill-posedness. The error estimate of Algorithm 1 may be carried out by choosing the length and angle sets following the analysis in \cite{BLRX2015,BT2010}. We intend to report such results in a separate work.


\section{Numerical examples}\label{se_numer}

In this section, we provide some numerical examples verifying the efficiency of the above Algorithm 1. The main computational costs in Algorithm 1 is Step 5 for the forward problem (\ref{eqn:prob}) and Step 9 for the inverse Fourier transform. To avoid the inverse crime, we will use fine grids (for instance, $100 \times 100$ or $200 \times 200$ equal-distance points) for the forward problem and a coarse grid ($90 \times 90$ equal-distance points) for the inversion both in the square domain $[-1,1]^{2}$.

\subsection{CASE 1}
In the first case, a circular domain $\Omega$ with a radius $r = 0.7$ and a simple true potential function $c$ are chosen, which are shown in the upper left panel of Figure \ref{fig:1_potential}. The potential function contains one peak and one valley patterns in the bounded domain $\Omega$. The boundary measurement points on $\partial \Omega$ are also presented in the upper left panel of Figure \ref{fig:1_potential}, marked by "$*$". The forward problem grid is $100 \times 100$ equal-distance points and the inversion grid is $90 \times 90$ equal-distance points in the domain $[-1,1]^{2}$. In Subsection \ref{subse_4.1.3}, we will show that one can improve the resolution of the reconstructed potential function if a finer grid of the forward problem is chosen.

\begin{figure}[htbp]
\centering
\includegraphics[width=0.8\textwidth]{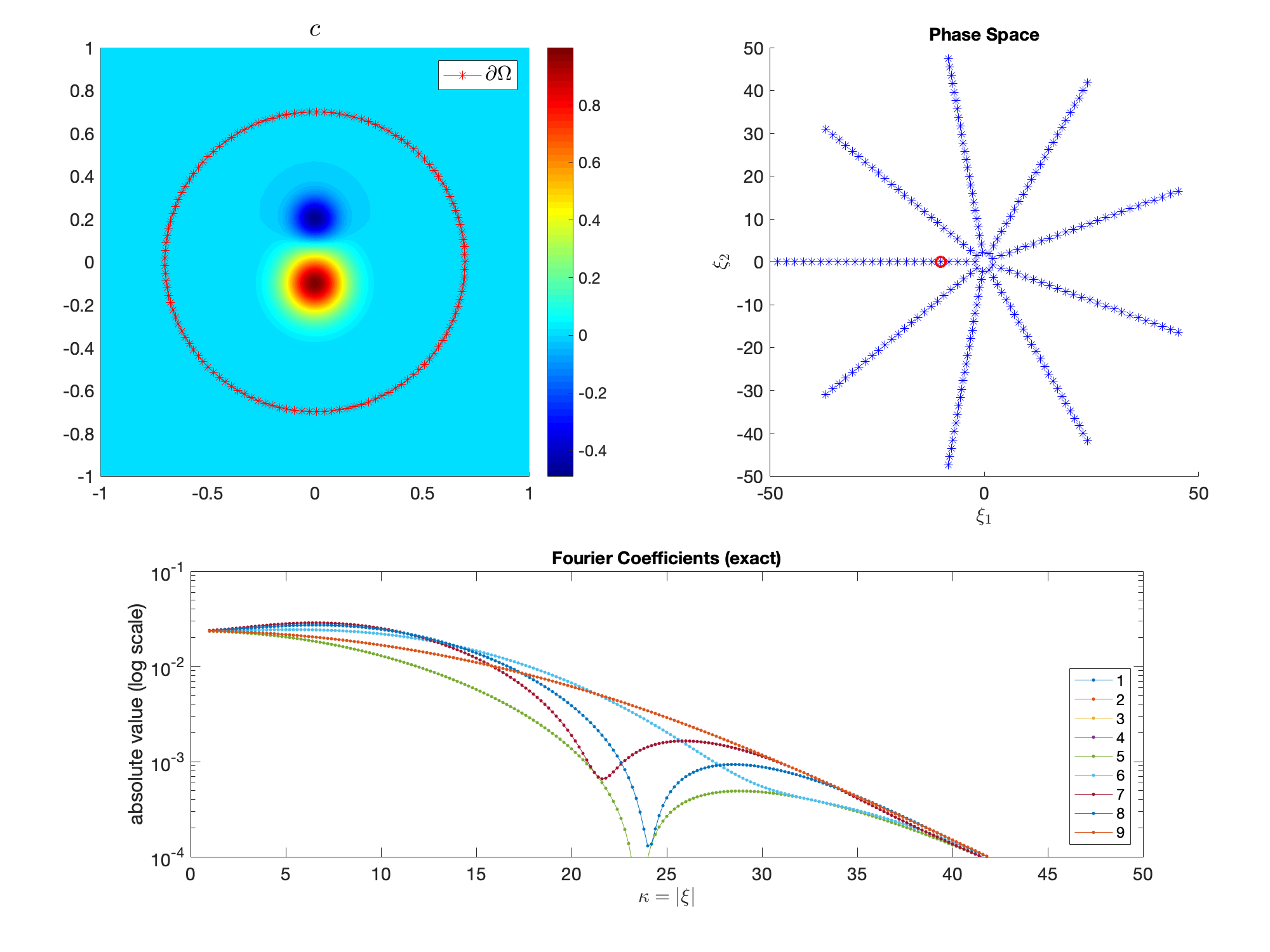}
\caption{Upper left: the true potential function $c$. Upper right: the nine slope lines in the phase space. The red "$\circ$" represents a sample point $\xi = \kappa \hat{y}$ with $\kappa = 10.2$ and $\hat{y} = (\cos(\pi),\sin(\pi)) = (-1,0)$. Bottom: Fourier coefficients of the true potential function near the nine slope lines in the phase space.}
\label{fig:1_potential}
\end{figure}

\subsubsection{CASE 1: General recovery}
To realize the Fourier transform $\mathcal{F}[\cdot]$, we use the following sampling points $\xi$ in the phase space which are marked by "$*$" in the upper right panel of Figure \ref{fig:1_potential}. The modulus of $\xi$, that is $\kappa = \left| \xi \right|$, are equally distributed in the interval $[1,50]$, and the step size is $0.2$. The degree of the angle between two adjacent slope lines is $\frac{2}{9} \pi$. For example, the red "$\circ$" in the upper right panel of Figure \ref{fig:1_potential} is a point in the phase space with an angle $\pi$ and a modulus $10.2$. Thus, the coordinate of this point $\xi$ in the phase space is $(-10.2,0)$, or $\kappa = 10.2$ and $\hat{y} = (\cos(\pi),\sin(\pi)) = (-1,0)$, i.e., $\xi = \kappa \hat{y}$. Fourier coefficients of the true potential function near the nine slope lines are presented in the bottom panel of Figure \ref{fig:1_potential}.

By varying different $\hat{y}$ (or corresponding $\hat{z}$) and different $\kappa = \left| \xi \right|$ in the given sets, we obtain different exponential solutions (\ref{eqn:numer_u0}) (or (\ref{eqn:numer_u0_xi})). In particular, we present two different $u_{0}$ (real part only) in Figures \ref{fig:u0g0_k15p2_xi8p4} for $\kappa \leqslant 2k$ and Figure \ref{fig:u0g0_k15p2_xi32p6} for $\kappa > 2k$ while $k = 15.2$ is fixed. The real part of the corresponding Dirichlet boundary data $g_{0}$ for the linearized problem are also presented in these two figures.
For instance, in Figure \ref{fig:u0g0_k15p2_xi8p4} (left), the exponential solution $u_{0}$ with a small length $\kappa = 8.4$ and $\hat{y} = (-0.17,0.98)$ is a plane wave indeed. In Figure \ref{fig:u0g0_k15p2_xi8p4} (middle), we present its Dirichlet boundary data $g_{0}$.
By substituting the Dirichlet boundary data $g_{0}$ into the original problem (\ref{eqn:prob}), we obtain the solution $u$ and its Neumann boundary data $g_{1} = \partial_{\nu} u$, which are shown in Figure \ref{fig:u0g0_k15p2_xi8p4} (right).
In order to obtain the measurable data $g_{1}^{\,\prime}$ of the linearized DtN map (\ref{eqn:dtn_p}), we subtract the benchmark solution $u_{0}$ with respect to the solution $u$ in Figure \ref{fig:1_u1g1p_k15p2_xi8p4} (left), and measure the linearized data $g_{1}^{\,\prime} = g_{1} - \partial_{\nu} u_{0}$ on the boundary $\partial \Omega$, which is shown in Figure \ref{fig:1_u1g1p_k15p2_xi8p4} (right).

\begin{figure}[htbp]
\centering
\includegraphics[width=1.0\textwidth]{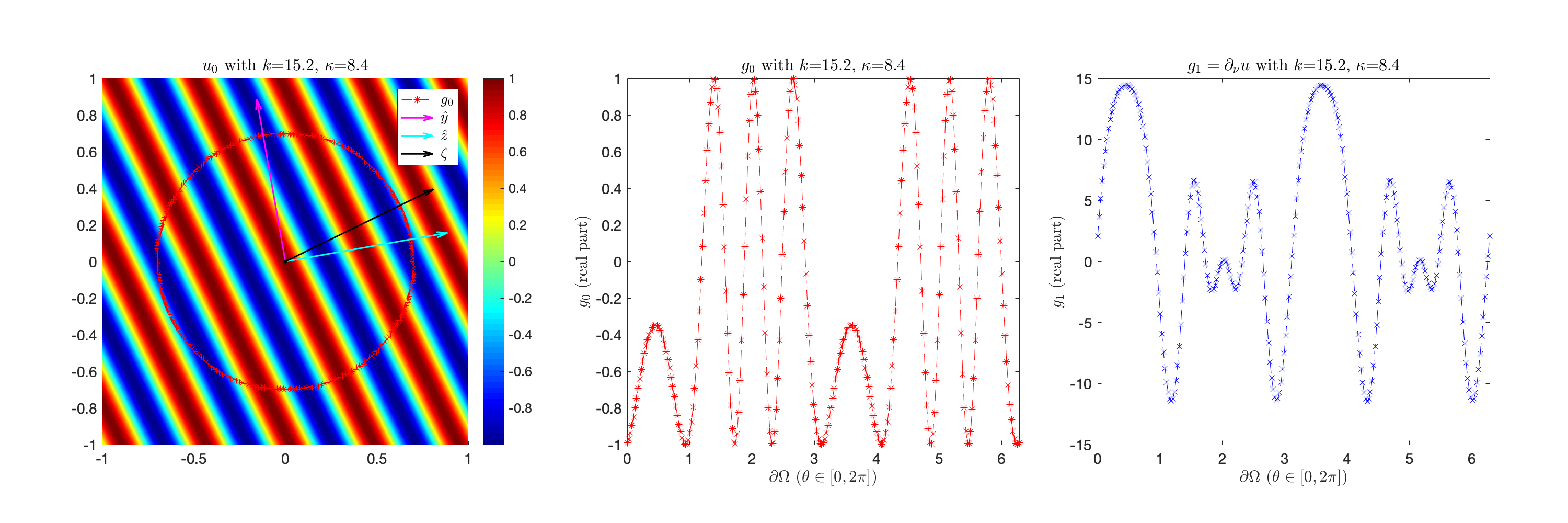}
\caption{Set $k = 15.2$, $\kappa = 8.4$ and $\hat{y} = (-0.17,0.98)$. Left: the exponential solution $u_{0}$. Middle: the Dirichlet boundary data $g_{0}$. Right: the Neumann boundary data $g_{1} = \partial_{\nu} u$.}
\label{fig:u0g0_k15p2_xi8p4}
\end{figure}

\begin{figure}[htbp]
\centering
\includegraphics[width=0.7\textwidth]{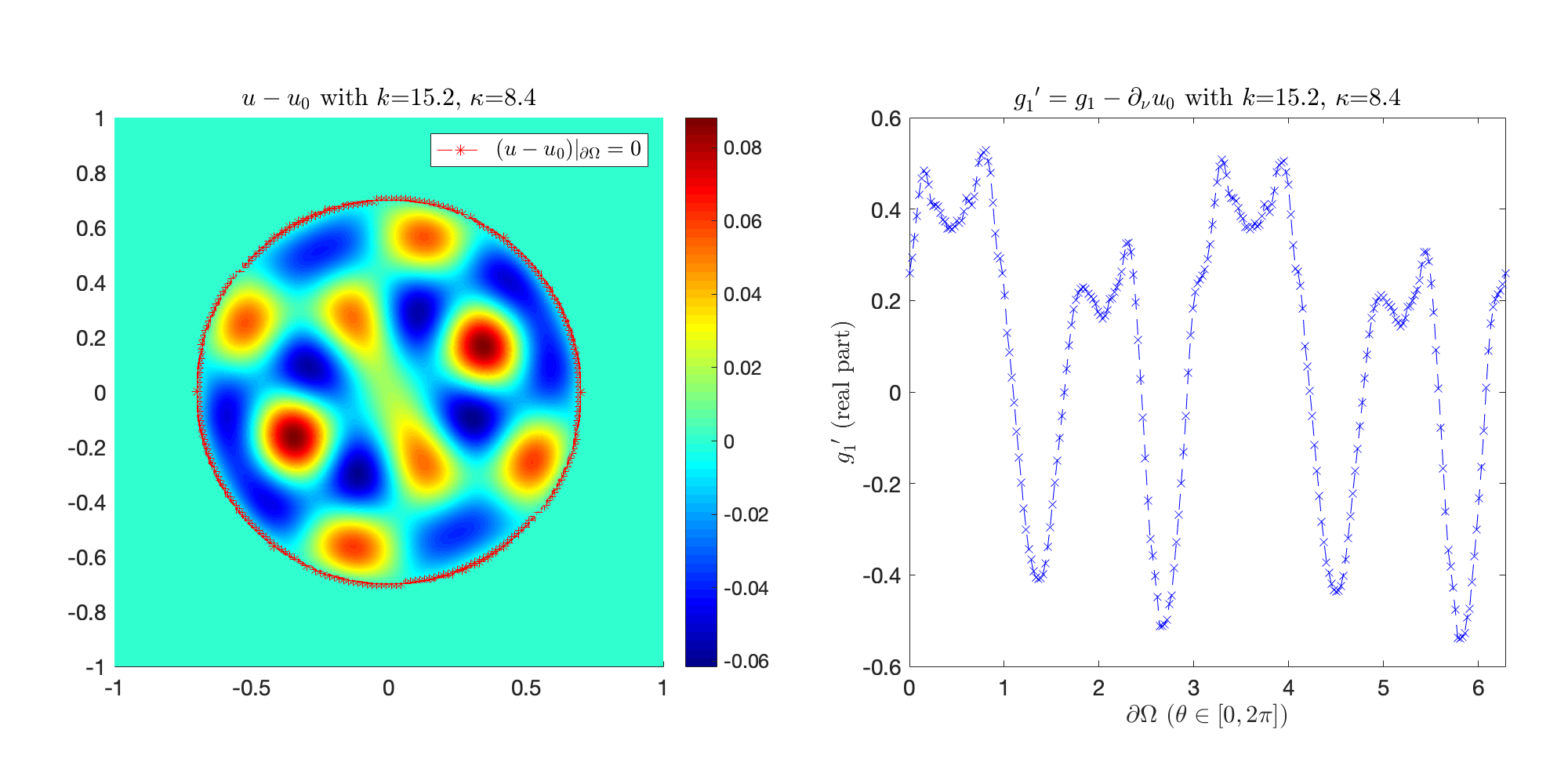}
\caption{Set $k = 15.2$, $\kappa = 8.4$ and $\hat{y} = (-0.17,0.98)$. Left: the difference $u - u_{0}$. Right: the linearized Neumann boundary data $g_{1}^{\,\prime} = g_{1} - \partial_{\nu} u_{0}$.}
\label{fig:1_u1g1p_k15p2_xi8p4}
\end{figure}

The second exponential solution $u_{0}$ with a large length $\kappa = 32.6$ and $\hat{y} = (-0.77,-0.64)$ is shown in Figure \ref{fig:u0g0_k15p2_xi32p6} (left) where exponential decay in the direction $\hat{z}$ is observed. The Dirichlet boundary data $g_{0}$ and the Neumann boundary data $g_{1} = \partial_{\nu} u$ are shown in Figure \ref{fig:u0g0_k15p2_xi32p6} (middle and right) respectively. The difference between $u$ and $u_{0}$, the measurable data $g_{1}^{\,\prime} = g_{1} - \partial_{\nu} u_{0}$ of the linearized DtN map are shown in Figure \ref{fig:1_u1g1p_k15p2_xi32p6}. Compared with the moderate linearized difference for $\kappa = 8.4$ as shown in Figure \ref{fig:1_u1g1p_k15p2_xi8p4}, we observe a large amplitude of the linearized difference in the right panel of Figure \ref{fig:1_u1g1p_k15p2_xi32p6} for $\kappa = 32.6$. Such phenomena further yields an unstable calculation of the Fourier coefficients when $\kappa > 2k$.

\begin{figure}[htbp]
\centering
\includegraphics[width=1.0\textwidth]{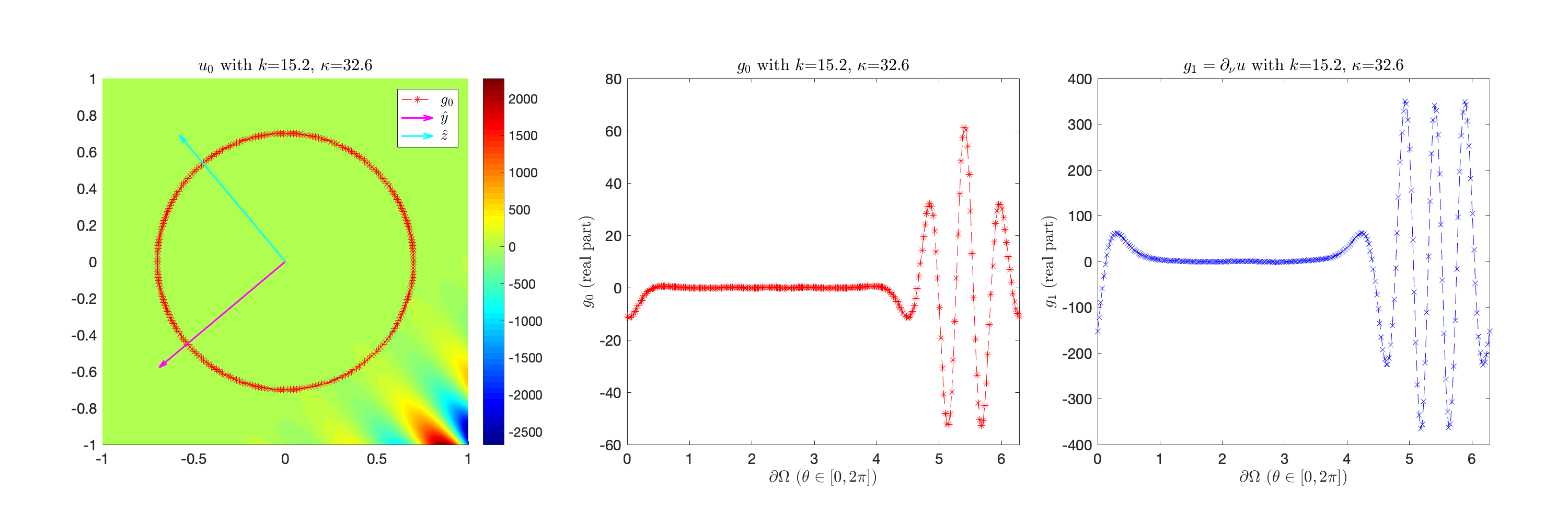}
\caption{Set $k = 15.2$, $\kappa = 32.6$ and $\hat{y} = (-0.77,-0.64)$. Left: the exponential solution $u_{0}$. Middle: the Dirichlet boundary data $g_{0}$. Right: the Neumann boundary data $g_{1} = \partial_{\nu} u$.}
\label{fig:u0g0_k15p2_xi32p6}
\end{figure}

\begin{figure}[htbp]
\centering
\includegraphics[width=0.7\textwidth]{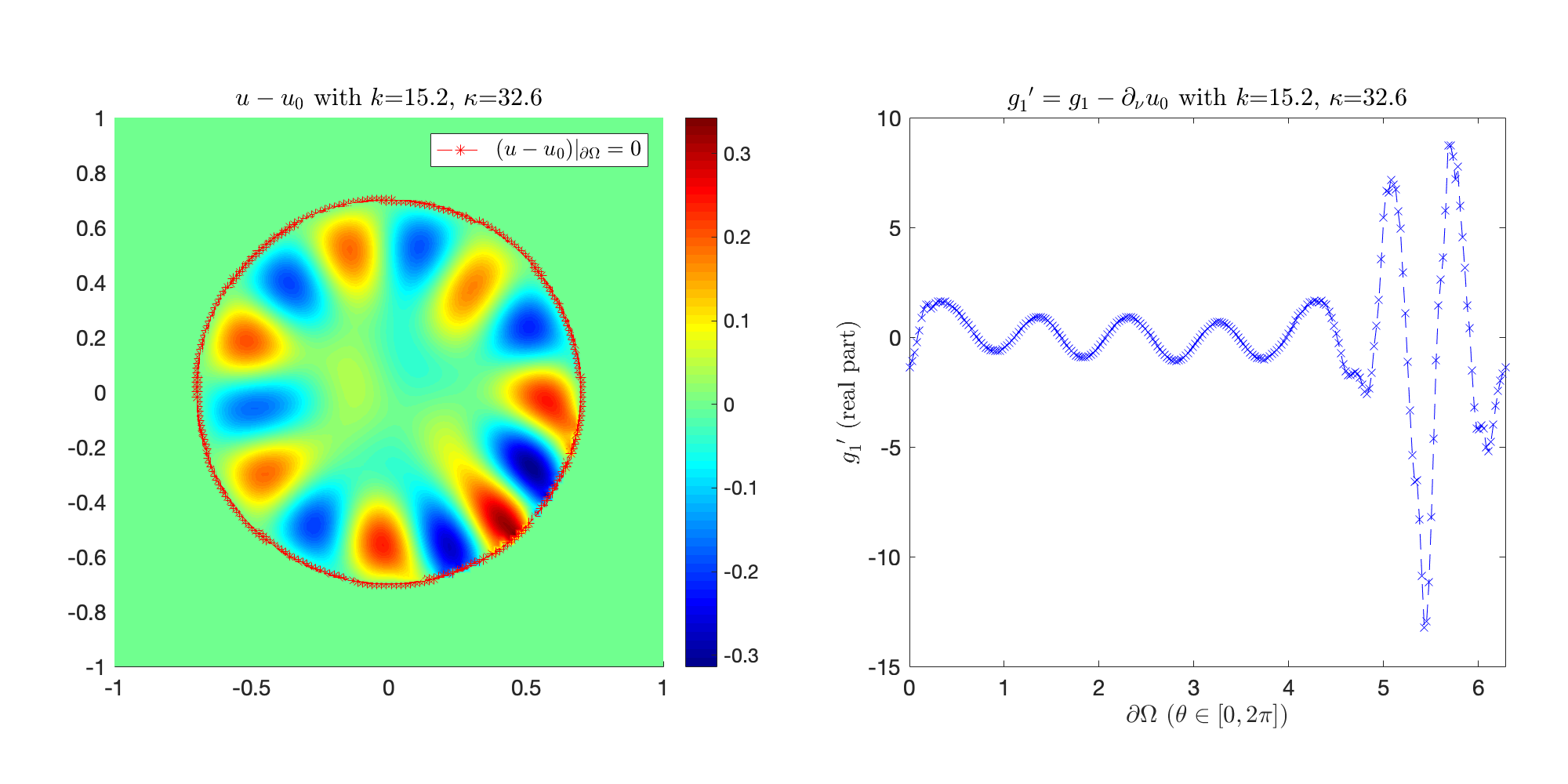}
\caption{Set $k = 15.2$, $\kappa = 32.6$ and $\hat{y} = (-0.77,-0.64)$. Left: the difference $u - u_{0}$. Right: the linearized Neumann boundary data $g_{1}^{\,\prime} = g_{1} - \partial_{\nu} u_{0}$.}
\label{fig:1_u1g1p_k15p2_xi32p6}
\end{figure}

Noticing that the exponential solution $v$ in (\ref{eqn:numer_v}) is similar to $u_{0}$, we use the same sampling points $\xi$ in the upper right panel of Figure \ref{fig:1_potential} for $v$. However, to illustrate the difference, for any given wavenumber $k$ and vector $\xi = \kappa \hat{y}$, the directions of two wave-vectors $\zeta = \frac{\kappa}{2} \hat{y} + \sqrt{k^{2} - \frac{\kappa^{2}}{4}} \, \hat{z}$ and $\zeta^{*} = \frac{\kappa}{2} \hat{y} - \sqrt{k^{2} - \frac{\kappa^{2}}{4}} \, \hat{z}$ are not the same. Indeed, since these vectors obey $\hat{y} \cdot \hat{z} = 0$, the wave-vectors $\zeta$, $\zeta^{*}$ of exponential solutions $u_{0}$ and $v$ are symmetric with respect to the vector $\xi$ (or $\hat{y}$). Thus, we could reflect the pattern of the exponential solution $u_{0}$ with respect to the vector $\xi$ to obtain the pattern of the solution $v$.

The key step of the reconstruction Algorithm 1 is to compute the value of each Fourier coefficient $\mathcal{F}[c](\xi)$ with different $\xi = \kappa \hat{y}$ for any fixed $k$ by employing the formula (\ref{eqn:integral_F_v}). The numerical coefficients are collected in Figure \ref{fig:1_fourier_k15p2} when $k = 15.2$. Each curve there represents the recovered value of Fourier coefficients on each slope of sampling points $\xi$ in the upper right panel of Figure \ref{fig:1_potential}, while the horizontal axis is the length interval $(0,50]$ for $\kappa = \left| \xi \right|$.
We shall emphasize that the chosen wavenumber $k = 15.2$ avoids the eigenvalues for the Laplacian operator in the domain $\Omega$. On the other hand, by choosing a wavenumber $k$ near the eigenvalues, we may recover a potential function with lower resolution which will be presented in next subsection.

\begin{figure}[htbp]
\centering
\includegraphics[width=0.8\textwidth]{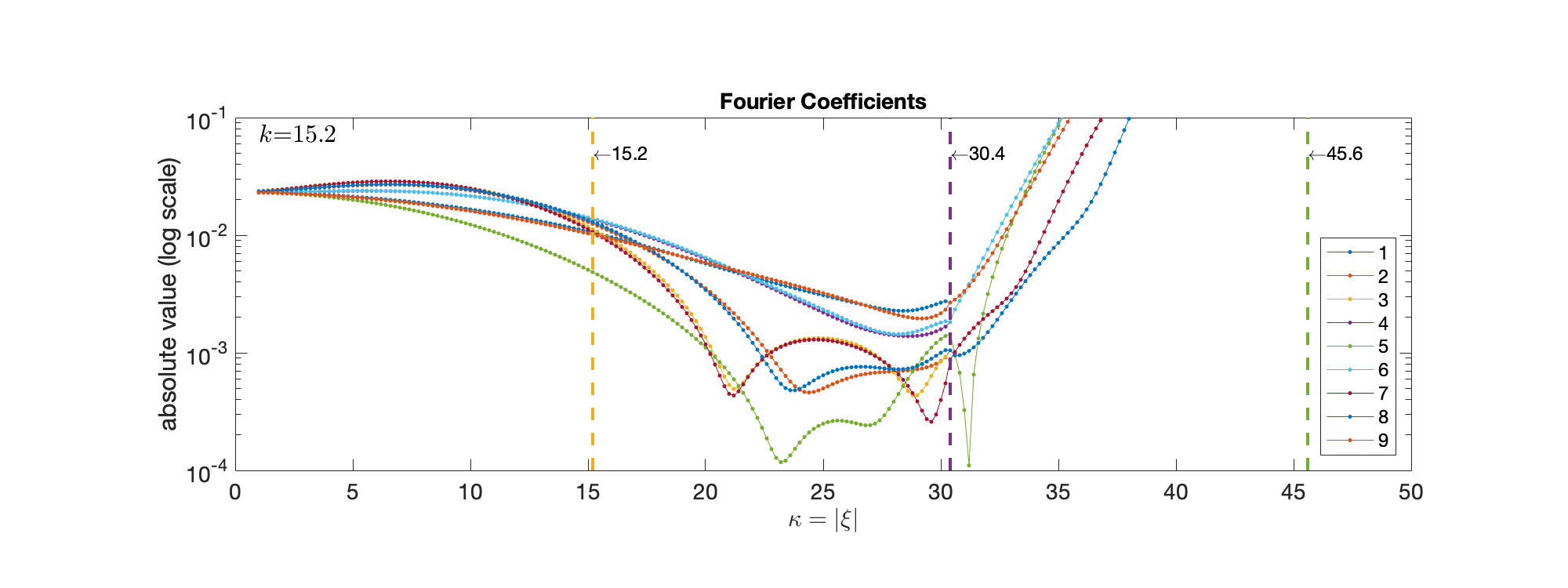}
\caption{Set $k = 15.2$. The recovered Fourier coefficients $\mathcal{F}[c](\xi)$ near the nine slope lines in the phase space.}
\label{fig:1_fourier_k15p2}
\end{figure}

Finally, we implement the inverse Fourier transform to reconstruct the potential function $c$. To obtain such a reconstruction, we need to choose a truncated value $K$ in the phase space such that all the Fourier coefficients with $\kappa \leq K$ shall be used to reconstruct an approximant.
By choosing different truncated value $K = m k$ while $m = 1, 2, 3$, we collect those reconstructed functions $c_{\rm inv}$ in Figure \ref{fig:1_c_inv_k15p2}. More precisely, the dashed lines in Figure \ref{fig:1_fourier_k15p2} highlight those threshold values as $k = 15.2$, $2 k = 30.4$ and $3 k = 45.6$ when $k = 15.2$. The Fourier coefficients $\mathcal{F}[c](\xi)$ with lengths $\kappa = \left| \xi \right|$ smaller than the threshold values, i.e. $\kappa \in (0,m k\,]$, will be used to generate $c_{\rm inv}$ respectively.
As one can observe, we obtain a stable reconstructed potential function $c_{\rm inv}$ with threshold values $k$ and $2 k$ in Figure \ref{fig:1_c_inv_k15p2} (left and middle), because all the Fourier coefficients $\mathcal{F}[c](\xi)$ are rather small for $\kappa \leqslant k$ or $\kappa \leqslant 2k$. The more Fourier coefficients when $\kappa \leqslant 2k$, the better resolution of the reconstructed potential $c_{\rm inv}$.
On the other hand, if we increase the truncated value to $3 k$, we observe some high frequency patterns in the right panel of Figure \ref{fig:1_c_inv_k15p2}, which are also reflected in the (blowing-up) Fourier coefficients when $\kappa > 2k$, see Figure \ref{fig:1_fourier_k15p2}, or referring to the estimate (\ref{eq_nminus1bound}) in the proof of Theorem \ref{thm_1}. In current work, we change different truncated value $K=mk$, $m=1,2,3$ but in real calculation we suggest to choose the truncated value $K=2k$.

\begin{figure}[htbp]
\centering
\includegraphics[width=1.0\textwidth]{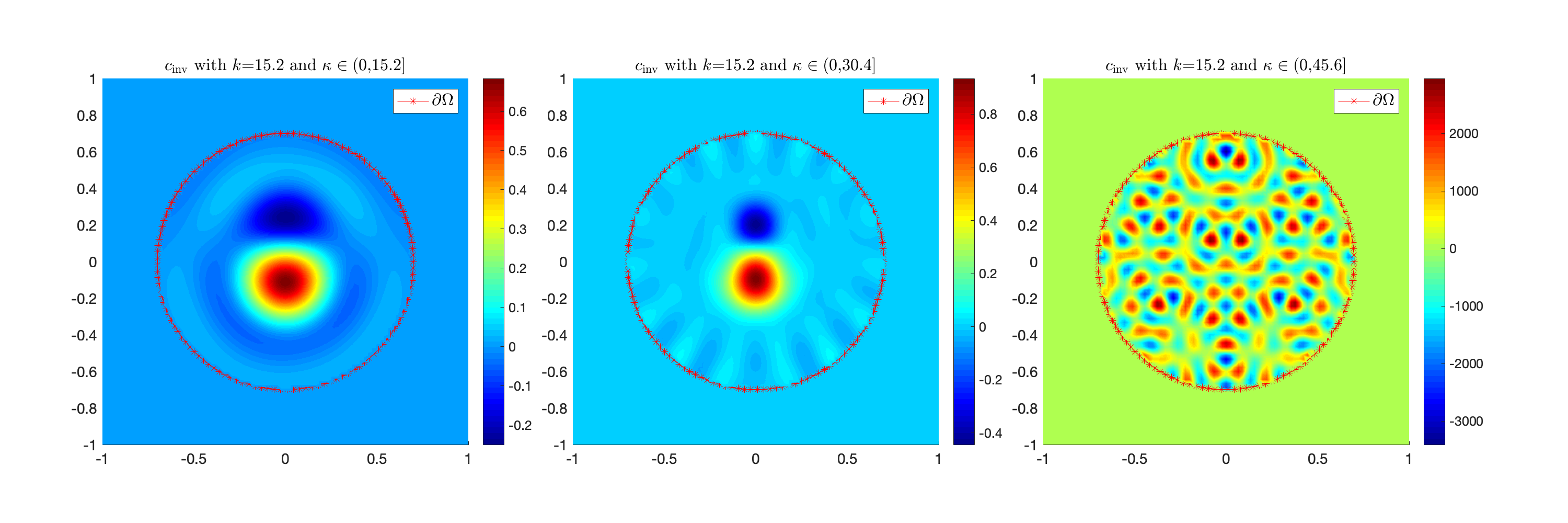}
\caption{Set $k = 15.2$. The reconstructed potential function $c_{\rm inv}$ with different truncated value $K = m k$ and $m = 1,2,3$. Left: $K = 15.2$. Middle: $K = 30.4$. Right: $K = 45.6$.}
\label{fig:1_c_inv_k15p2}
\end{figure}

We shall note that a large value of the wavenumber $k$ allows us to use more (stable) Fourier coefficients in the phase space when $\kappa \leqslant 2k$. To visualize such difference, we presents the reconstructed potential function $c_{\rm inv}$ in Figure \ref{fig:1_c_inv_k5} by choosing $k = 5$ and the truncated values $K = 5, 10, 15$ respectively.

\begin{figure}[htbp]
\centering
\includegraphics[width=1.0\textwidth]{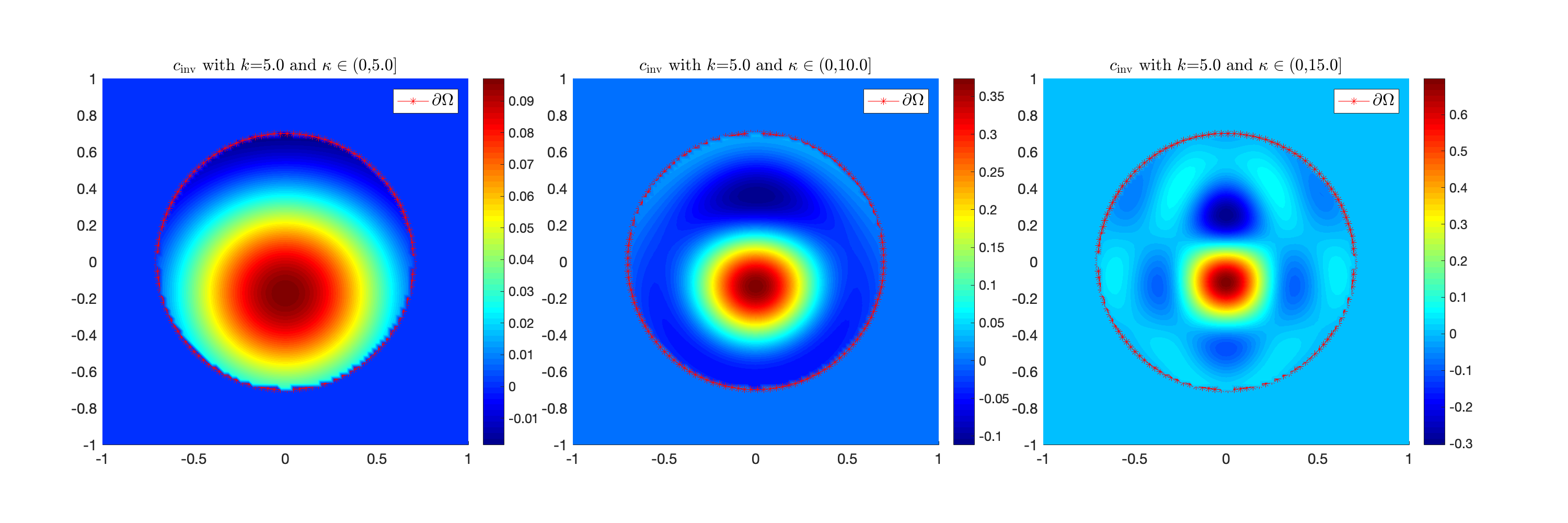}
\caption{Set $k = 5$. The reconstructed potential function $c_{\rm inv}$ with different truncated value $K = m k$ and $m = 1,2,3$. Left: $K = 5$. Middle: $K = 10$. Right: $K = 15$.}
\label{fig:1_c_inv_k5}
\end{figure}


\subsubsection{CASE 1: Recovery at a large wavenumber near eigenvalues}

In this subsection, we numerically investigate the consequence by choosing a large wavenumber near eigenvalues for a Laplacian operator in the circular domain $\Omega$.

A particular choice of such a large wavenumber is $k = 12.3625$. In Figure \ref{fig:1_c_inv_k12p4}, we present the reconstructed Fourier coefficients and the reconstructed potential function $c_{\rm inv}$ analogously as in the previous subsection by choosing truncated values $K = 2k = 24.7250$. Noticing that the high frequency patterns appear in Figure \ref{fig:1_c_inv_k12p4} (right), we also observe (blowing-up) Fourier coefficients compared with those in Figures \ref{fig:1_fourier_k15p2}.

\begin{figure}[htbp]
\centering
\includegraphics[width=1.0\textwidth]{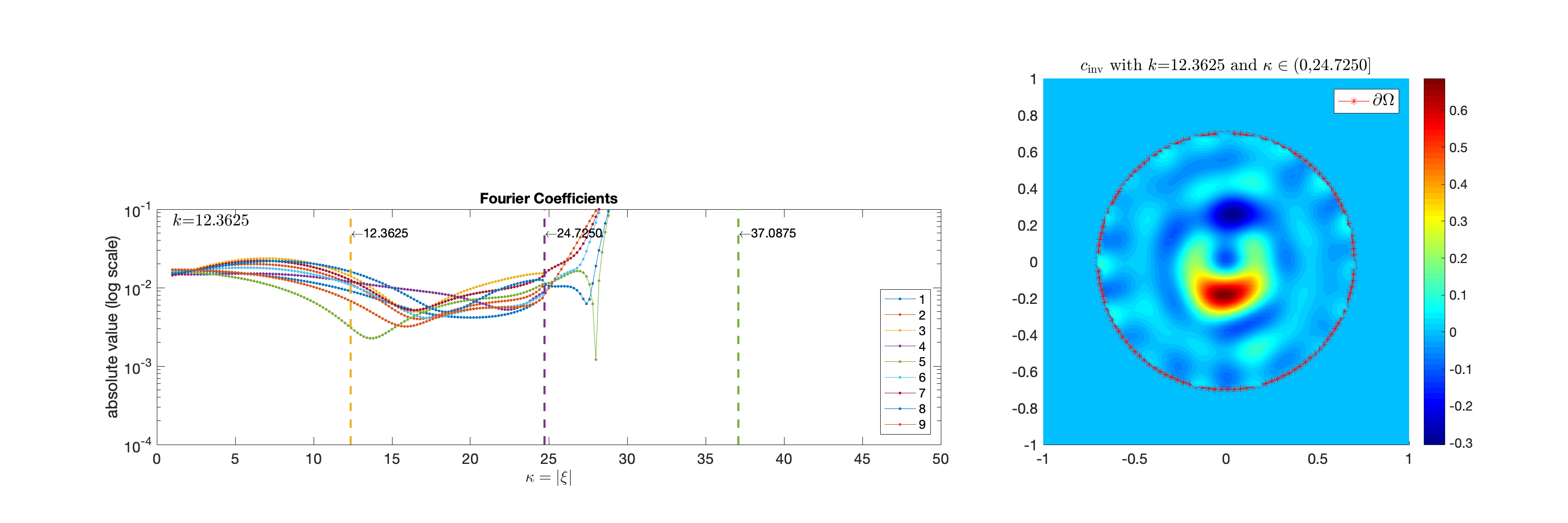}
\caption{Recovery by using the wavenumber $k = 12.3625$. Left: the recovered Fourier coefficients $\mathcal{F}[c](\xi)$ near the nine slope lines in the phase space. Right: the reconstructed potential function $c_{\rm inv}$ with truncated value $K = 24.7250$.}
\label{fig:1_c_inv_k12p4}
\end{figure}


\subsubsection{CASE 1: Finer grids and less Fourier modes}\label{subse_4.1.3}

We further focus on another two numerical aspects enhancing or weakening the resolution of the reconstructed potential function $c_{\rm inv}$.

The first aspect is the grid of the forward problem. In the above two subsections, we choose a $100 \times 100$ equal-distance points grid in $[-1, 1]^{2}$. The chosen grid may capture some solution patterns if the wavenumber $k$ is small. But when the wavenumber $k$ increases, for instance from $5$ to $15.2$, the high frequency patterns induced by large wavenumbers may not be accurately demonstrated in the grid. To weaken such difficulty, we consider a finer $200 \times 200$ equal-distance points grid in $[-1, 1]^{2}$. By implementing the same approach in the above subsection, we present the Fourier coefficients and their reconstructed potential function $c_{\rm inv}$ in Figure \ref{fig:1_c_inv_k15p2_finergrid} for $k = 15.2$ again. As one can observe, the Fourier coefficients of the finer grid become more accurate in the large length interval ($\kappa \in [20,30]$) than those of the original grid. The reconstructed potential functions $c_{\rm inv}$ also enhance the resolution (in the peak pattern) if we choose the truncated value $K = 2 k = 30.4$, see Figure \ref{fig:1_c_inv_k15p2_finergrid} (right).

\begin{figure}[htbp]
\centering
\includegraphics[width=1.0\textwidth]{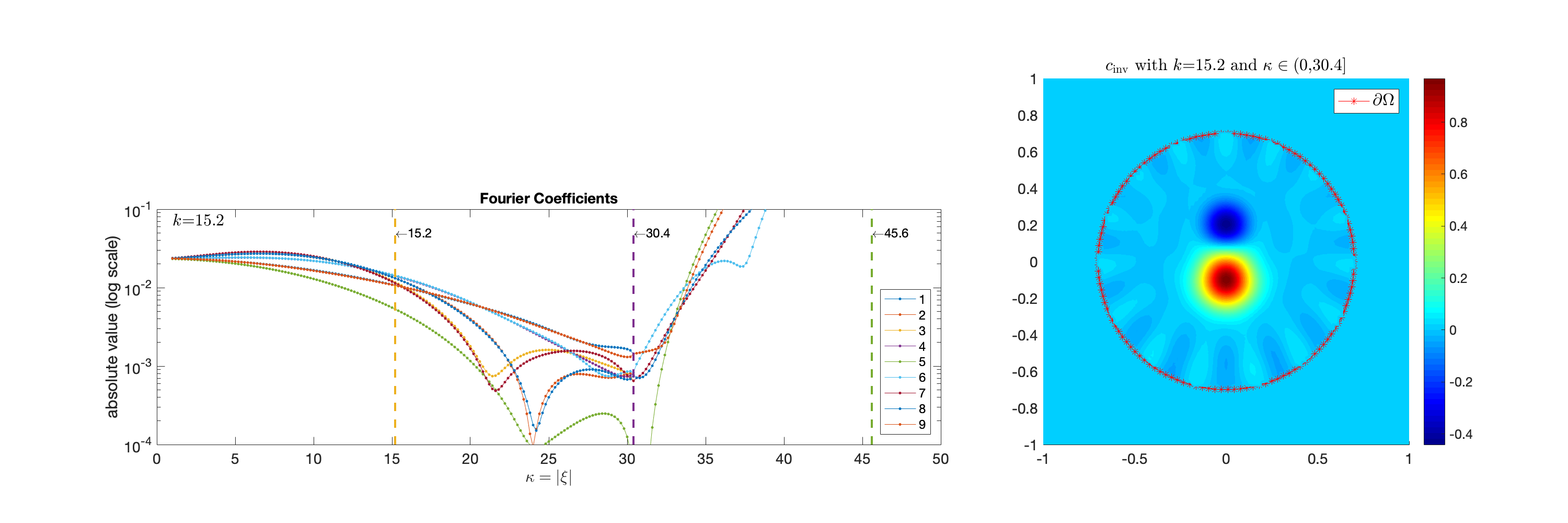}
\caption{Finer grids. Set $k = 15.2$. Left: the recovered Fourier coefficients $\mathcal{F}[c](\xi)$. Right: the reconstructed potential function $c_{\rm inv}$ with truncated value $K = 30.4$.}
\label{fig:1_c_inv_k15p2_finergrid}
\end{figure}

We also investigate the influence of the observation angles, i.e. the limited-angle slope lines in the upper right panel of Figure \ref{fig:1_potential}. In Figure \ref{fig:1_c_inv_k15p2_finergrid_phase}, we show the recovered potential functions with $2$, $3$, $7$ different slope lines by choosing the truncated value $K = 2 k$ and $k = 15.2$. Compared with the result in Figure \ref{fig:1_c_inv_k15p2_finergrid}, we observe that the more slope lines, the better accurately recovered potential functions, because of the gained Fourier coefficients. In principle, one can include more angle slope lines to obtain better resolution.

\begin{figure}[htbp]
\centering
\includegraphics[width=1.0\textwidth]{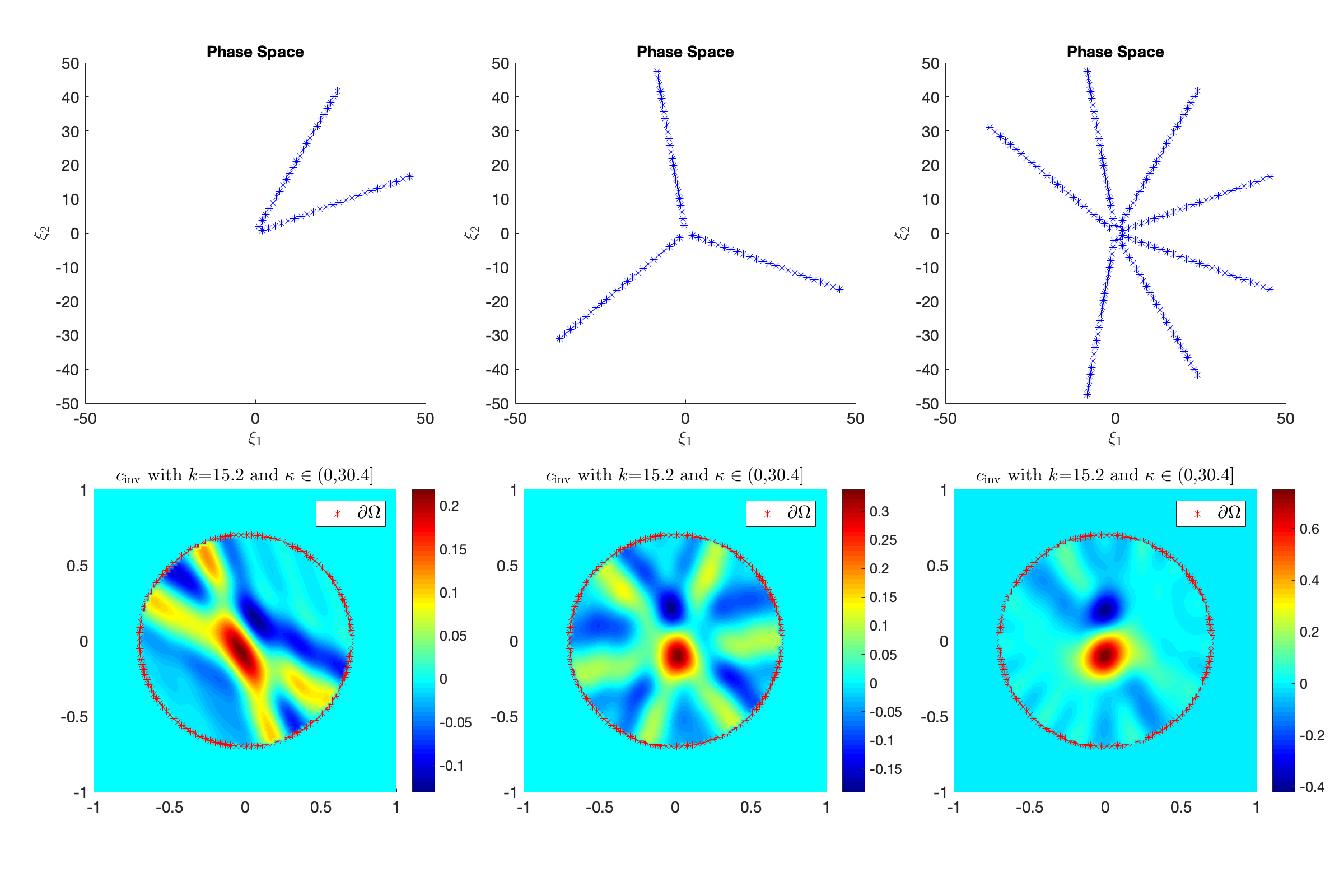}
\caption{The reconstructed potential function $c_{\rm inv}$ with limited-anlges. Set $k = 15.2$.}
\label{fig:1_c_inv_k15p2_finergrid_phase}
\end{figure}


\subsection{Another numerical example}

Another example considers a complicate potential function $c$ in Figure \ref{fig:2} (left) with the same circular domain $\Omega$. Similar to the previous example, the same set of the sampling points $\xi$ in the phase space is utilized referring to the upper right panel of Figure \ref{fig:1_potential}. For the sake of simplicity, we skip all the discussion on the reconstruction algorithm but provide the reconstructed potential function $c_{\rm inv}$ with the truncated value $K=2 k$ and $k=5.0$, $k = 20.0$ respectively in Figure \ref{fig:2} (middle \& right). Similarly, we also have chosen a finer $200 \times 200$ equal-distance points grid in the domain $[-1, 1]^{2}$ for the forward problem. Comparing the reconstructed potential functions $c_{\rm inv}$ for $k=5$ and $k=20$, we clearly visualize the improved resolution at the large wavenumber.

\begin{figure}[htbp]
\centering
\includegraphics[width=1.0\textwidth]{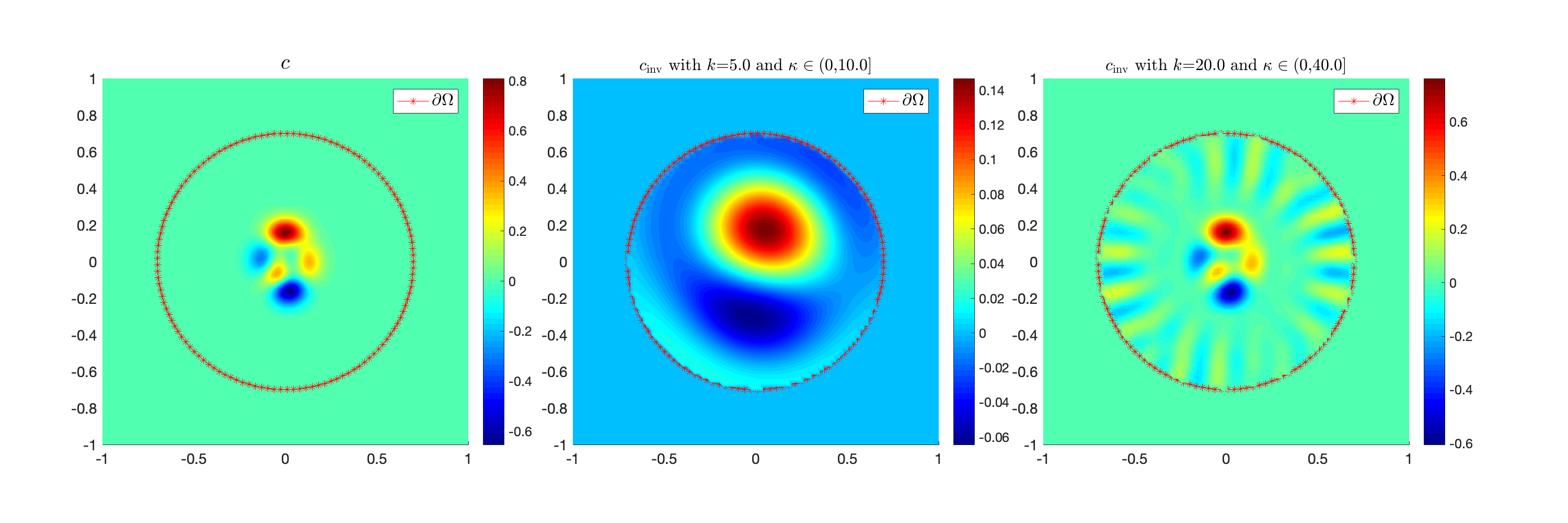}
\caption{Left: The true potential function $c$. Middle and right: its reconstruction $c_{\rm inv}$ with finer grid, and the chosen truncated wavenumber is $2 k$ while $k = 5.0$ and $20.0$, respectively.}
\label{fig:2}
\end{figure}


\section{Conclusion}\label{se_con}
The original inverse problem for the Schr\"{o}dinger potential is ill-posed and nonlinear (moreover, non-convex). To study stability we prefer to focus on the most serious difficulty: ill-conditioning and we linearized inverse problem to avoid additional difficulties with multiple local minima. A  numerical solution of the linearized problem is much faster and more reliable and is quite satisfactory in many applications.

We demonstrated that recovery of the Schr\"{o}dinger potential from all boundary data at a fixed energy/wavenumber is dramatically improving for a larger $k$. There is no doubt that this improvement will be more significant when finding more complicated $c$ or using a complete nonlinear problem.

These results promise a better numerical reconstruction of the conductivity coefficient in the stationary Maxwell system at higher wavenumbers, as predicted analytically in \cite{ILW2016}. We intend to work on the reconstruction at least in the linearized version in near future. There are good expectations that such numerical results will be a solid base for a serious improvement in the electrical impedance tomography which now suffers from a very low resolution. As known, this type of tomography has many geophysical and medical applications.

\bibliographystyle{amsplain}

\end{document}